\documentclass[reqno]{amsart}
\usepackage[foot]{amsaddr}

\usepackage[hidelinks]{hyperref}
\usepackage{url}

\usepackage{amssymb}
\usepackage{amscd}
\usepackage{amsthm}
\usepackage{setspace}
\usepackage{enumerate}
\usepackage{graphicx}
\usepackage{mathtools}
\usepackage{tcolorbox} 
\usepackage{subcaption} 
\usepackage{siunitx}
\usepackage{tikz-cd}
\usepackage{bm,blkarray}
\numberwithin{equation}{section}
\usepackage{algorithm}
\usepackage{algpseudocode}

\usepackage[nameinlink]{cleveref}

\usepackage{mathtools}
\usepackage[tableposition=top]{caption}
\usepackage{booktabs,dcolumn}

\usepackage{refcount}

\renewcommand\a{\mathbf{a}}

\newcommand\C{\mathbb{C}}

\newcommand\dist{{\operatorname{dist}}}

\newcommand\eps{\varepsilon}
\renewcommand\epsilon\varepsilon

\renewcommand{\SS}{\mathbb S}

\newcommand{\sph}[1]{ {\mathbb S^{#1}} }

\newcommand{\ks}{\textnormal{KS}}

\usepackage{xcolor}      %
\usepackage{listings}    %

\definecolor{codegreen}{rgb}{0,0.6,0}
\definecolor{codegray}{rgb}{0.5,0.5,0.5}
\definecolor{codepurple}{rgb}{0.58,0,0.82}
\definecolor{white}{rgb}{1,1,1}

\lstdefinestyle{mystyle}{
    backgroundcolor=\color{white}, 
    commentstyle=\color{codegreen},
    keywordstyle=\color{magenta}, numberstyle=\tiny\color{codegray}, stringstyle=\color{codepurple}, basicstyle=\ttfamily\footnotesize, breakatwhitespace=false,  breaklines=true,                 
    captionpos=b,                    
    keepspaces=true,                 
    numbers=left,                    
    numbersep=5pt,                  
    showspaces=false,                
    showstringspaces=false,
    showtabs=false,                  
    tabsize=2,
    frame=single,
    rulecolor=\color{black},
    title=\lstname
}

\title[Condition number of Fourier submatrices]{On the exponential rate of the condition number of Fourier submatrices and Vandermonde matrices}

\author{Rikhav Shah}
\address{Department of Mathematics, Massachusetts Institute of Technology, Cambridge, MA, 02139 USA.}
\email{rdshah@mit.edu}
\author{John Urschel}
\email{urschel@mit.edu}

\newtheorem{theorem}{Theorem}[section]

\newtheorem{lemma}[theorem]{Lemma}

\newtheorem{proposition}[theorem]{Proposition}

\newtheorem{corollary}[theorem]{Corollary}

\newtheorem{remark}[theorem]{Remark}

\AddToHook{env/lemma/begin}{\crefalias{theorem}{lemma}}
\AddToHook{env/corollary/begin}{\crefalias{theorem}{corollary}}

\crefname{theorem}{Theorem}{Theorems}
\crefname{lemma}{lemma}{lemmas}
\crefname{corollary}{corollary}{corollaries}
\crefname{equation}{}{}
\crefname{section}{section}{sections}
\crefname{subsection}{subsection}{subsections}
\crefname{figure}{figure}{figures}
\crefname{proposition}{proposition}{propositions}
\DeclareMathOperator*{\argmin}{arg\,min}
\DeclareMathOperator*{\argmax}{arg\,max}
\DeclareMathOperator{\erf}{erf}

\usepackage{mleftright}
\newcommand{\abs}[1]{\mleft|#1\mright|}

\newcommand{\magn}[1]{\left\|#1\right\|}
\newcommand{\rmagn}[1]{\|#1\|}

\newcommand{\pare}[1]{\mleft(#1\mright)}

\newcommand{\set}[1]{{\left\{{#1}\right\}}}

\newcommand{\bmat}[1]{\begin{bmatrix}#1\end{bmatrix}}
\newcommand{\pmat}[1]{\begin{pmatrix}#1\end{pmatrix}}

\newcommand{\floor}[1]{\mleft\lfloor#1\mright\rfloor}

\newcommand{\wrt}{\,\textnormal d}

\newcommand{\spliteq}[2]{\begin{equation}#1\begin{split}#2\end{split}\end{equation}}
\newcommand{\eq}[1]{\begin{equation}{#1}\end{equation}}
\DeclareMathOperator{\ball}{Disk}
\newcommand{\sdist}{d_{\mathbb S^1}}
\newcommand{\inr}[2]{\left<#1,#2\right>}

\DeclareMathOperator{\Cl}{Cl_2}
\DeclareMathOperator{\supp}{supp}

\renewcommand{\a}{\bm a}

\begin{document}

\begin{abstract}
The discrete Fourier transform matrix is one of the most important matrices in linear algebra, and submatrices of it arise in a variety of applications. Though the discrete Fourier transform matrix is unitary, its submatrices can be exponentially ill-conditioned, an obstacle to accurate computation.
This work resolves the exact rate of the exponential ill-conditioning for square submatrices with contiguous rows and columns. As a consequence, we obtain a tight upper bound of $2 G/\pi$ on the exponential rate for all submatrices with contiguous columns, or, equivalently, all Vandermonde submatrices with distinct support points, where $G$ is Catalan's constant. These results follow from a more general analysis of Vandermonde and Vandermonde-like matrices for which exact estimates for exponential ill-conditioning are developed in terms of logarithmic potentials.

\end{abstract}

\maketitle

\section{Introduction}

The discrete Fourier transform matrix 
\[F = \left(e^{ \frac{2 \pi i (j-1)(k-1)}{N} } \right)_{j,k =1}^{N} \in \mathbb{C}^{N \times N},\]
where $i$ is the imaginary unit, is one of the most important matrices in linear algebra. Submatrices of $F$ appear in applications where only a specific subset of frequencies of a given signal is measured or processed. The most prevalent application is imaging \cite{b0,b1,b2,b3,b4,b5,b6}, 
though submatrices of the Fourier matrix also appear in many other settings, including wireless communication \cite{b7,b8,b9,b10}, Fourier extension methods \cite{b11,b12,b13,b14,b15,b16}, and finite frames \cite{b17,b18}. Unfortunately, while $F/\sqrt{N}$ is a unitary matrix, its submatrices are known to be numerically low-rank (see Edelman, McCorquodale, and Toledo \cite{b19}) and exponentially ill-conditioned. This exponential ill-conditioning was first shown for a submatrix of general shape by Moitra \cite{b20}, though this question has been studied by many authors in a variety of special cases, both before and after Moitra's work \cite{b11,b21,b0,b22,b23,b24,b25,b26,b27,b28,b29,b30}. Most notably, Barnett proved that the condition number of any $p \times q$ contiguous submatrix is at least $\exp\pare{\tfrac{\pi}{2}\pare{\min(p,q) - pq/N}}$ \cite{b31}. In addition, he performed numerical calculations which suggested that his bound was reasonably accurate for tall and skinny matrices but became quite inaccurate for square submatrices (see \Cref{a0}). Here, we obtain the exact exponential rates governing the condition number of contiguous, square $p \times p$ submatrices (see \Cref{a1}). In particular, we resolve not just the exponential rate of the condition number, but the polynomial factor as well. We note that the square submatrix setting is of special importance, as the condition number of a $p\times q$ contiguous submatrix is upper bounded by the condition number of a $\max(p,q)\times \max(p,q)$ contiguous submatrix.

Our analysis also relaxes the requirement that the entire submatrix be contiguous; we only require that the columns form a contiguous set. We omit studying completely arbitrary submatrices for two reasons. First, whenever $N$ is composite, there are singular submatrices (for instance, for $N = rs$, the $2 \times 2$ submatrix formed by the $1^{st}$ and $(r+1)^{th}$ rows and $1^{st}$ and $(s+1)^{th}$ columns is singular), but there are no singular submatrices with contiguous columns. Interestingly, when $N$ is prime, all submatrices are nonsingular by Chebotar\"ev's theorem \cite{b32}. Second, the submatrices with contiguous columns are precisely the matrices which are themselves Vandermonde matrices with distinct nodes, up to a permutation of the columns. In applications where Fourier submatrices appear, they invariably seem to be themselves Vandermonde matrices. In the setting of contiguous columns only, we again recover the correct exponential rate of the condition number, though we do not obtain the polynomial factor (\Cref{a2}).

We derive the aforementioned bounds from a more general analysis of Vandermonde matrices with nodes on the unit circle (which we further relax later in \Cref{a2}). Such matrices are often called non-uniform discrete Fourier transform matrices, see \cite[Section 1 \& Remark 2]{b33} for applications where such matrices arise and the influence of the condition number. Let $V(z_0,\ldots,z_n) =( z_k^{j} )_{j,k=0}^n$ be the square Vandermonde matrix with nodes $z_0,\ldots,z_n \in \mathbb{C}$, $\mathbb{S}^1 \subset \mathbb{C}$ be the unit circle, and $\sdist(z,w)$ be the arc length between $z,w \in \mathbb{S}^1$. We prove the following theorem.

\begin{theorem}\label{a3}
Let $z_0,\ldots,z_n \in \mathbb{S}^1$ and $ \min_{j \ne k} \sdist(z_j,z_k) \ge \eps>0$. Then
 \eq{\label{a4}
 \frac{\log\magn{V(z_0,\ldots,z_n)^{-1}}}n \le \frac4{n \eps} \int_{0}^{\frac{\eps n}4} \log \cot\phi\wrt\phi -\frac{C\log n}n \pm \frac{O(\log (n \eps(2\pi - n \eps)))}n}%
 for some constant $C \in \left[ \tfrac{1}{2}, 1\right]$. Furthermore, equality is achieved, with $C = 1$, when $z_j = e^{i j \eps}$ for $j \in \{0,\ldots,n\}$.
\end{theorem}
The proof of \Cref{a3} is surprisingly simple in spirit, using only Lagrange interpolating polynomials, logarithmic potentials, and Riemann summation. Let $[N] = \{1,\ldots,N\}$ and, for a matrix $A \in \mathbb{C}^{N \times N}$ and subsets $S,T \subset [N]$, let $A_{S,T}$ denote the $|S| \times |T|$ submatrix of $A$ with rows indexed by $S$ and columns indexed by $T$ and let $\kappa(A_{S,T}) = \sigma_1(A_{S,T})/\sigma_{\min(|S|,|T|)}(A_{S,T})$ be its condition number. A subset $S \subset [N]$ is said to be {\it cyclically contiguous} if it equals a set of consecutive integers modulo $N$. As a corollary to \Cref{a3}, we obtain a tight bound on how ill-conditioned a submatrix of the Fourier matrix with contiguous columns (or rows) can be.

\begin{corollary}\label{a1}
Let $F = \big(e^{ \frac{2 \pi i (j-1)(k-1)}{N} } \big)_{j,k =1}^{N} \in \mathbb{C}^{N\times N}$ be the Fourier matrix, $S \subset [N]$ be an arbitrary subset, $T \subset[N]$ be a cyclically contiguous subset, and $\alpha = \max(|S|,|T|)/N$. Then
\[\log \kappa \left(F_{S,T} \right) \le  \frac{2N}{\pi} \int_{0}^{\frac{\alpha \pi}{2}} \log \cot \phi\wrt\phi  \pm O(\log \alpha(1-\alpha)).\]
If both $S,T \subset[N]$ are cyclically contiguous subsets, then 
\[\log \kappa \left(F_{S,T} \right) \le  \frac{2N}{\pi} \int_{0}^{\frac{\alpha \pi}{2}} \log \cot \phi\wrt\phi - \frac{1}{2}\log (N) \pm O(\log \alpha(1-\alpha)),\]
and equality is achieved when $|S| = |T|$.
\end{corollary}

The quantity $\displaystyle{\frac{2}{ \pi} \int_{0}^{\frac{\alpha \pi}{2}} \log \cot \phi\wrt\phi}$ is maximal for $\displaystyle{\alpha = \frac{1}{2}}$, where it equals $\displaystyle{\frac{2 G}{\pi}}$, where $\displaystyle{G = \sum_{k=0}^\infty \frac{(-1)^k}{(2k+1)^2}}$ is Catalan's constant. As a result, we also have the following uniform upper bound.

\begin{corollary}\label{a5}
Let $F = \big(e^{ \frac{2 \pi i (j-1)(k-1)}{N} } \big)_{j,k =1}^{N} \in \mathbb{C}^{N\times N}$ be the Fourier matrix, $S \subset [N]$ be an arbitrary subset, and $T \subset[N]$ be a cyclically contiguous subset. Then
\[\log \kappa \left(F_{S,T} \right) \le  \frac{2GN}{\pi} \pm O(1),\]
where $G$ is Catalan's constant. If both $S,T \subset[N]$ are cyclically contiguous subsets, then
\[\log \kappa \left(F_{S,T} \right) \le  \frac{2GN}{\pi}  - \frac{1}{2} \log (N)  \pm O(1),\]
and equality is achieved when $|S| = |T|= \lceil N/2 \rceil$ and both $S,T \subset[N]$ are cyclically contiguous subsets.
\end{corollary}

Stated differently, the condition number of submatrices of the $N\times N$ discrete Fourier matrix with cyclically contiguous columns is asymptotically upper bounded by $(1.792)^N$, and there exists a submatrix with condition number asymptotically exceeding $(1.791)^N$, improves upon the lower bound of $(1.48)^N$ due to Barnett. We compare our estimates for square submatrices to those of Barnett \cite{b31} and numerically observed condition numbers for $N = 512$ in \Cref{a0}. Our estimates complement those of Barnett quite well, as \Cref{a1} is provably accurate for square matrices and Barnett's bound seems accurate for highly rectangular matrices. In addition, we stress that the bounds presented above are not purely asymptotic, but apply to practical values of $N$. The use of big Oh notation is purely for the sake of simplicity. Tracking the arguments contained in this paper, which do not aim to minimize the constant error term, produces a small constant. A more complicated analysis, using a refinement of the ideas presented here, would likely produce the correct constant for \Cref{a1}. Based on the experimental results in \Cref{a0}, this constant appears to be approximately $\approx 1.04$. We leave the identification and proof of this constant to the interested reader.

\begin{figure}[h!]
\centering
  \begin{subfigure}[t]{0.44\textwidth}
    \centering
    \includegraphics[width =\textwidth]{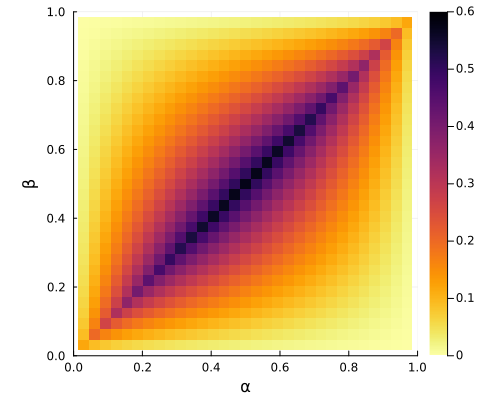}
    \caption{computed rates for $N = 512$}
    \label{a6}
  \end{subfigure}\hfill
  \begin{subfigure}[t]{0.53\textwidth}
    \centering
\includegraphics[width =\textwidth]{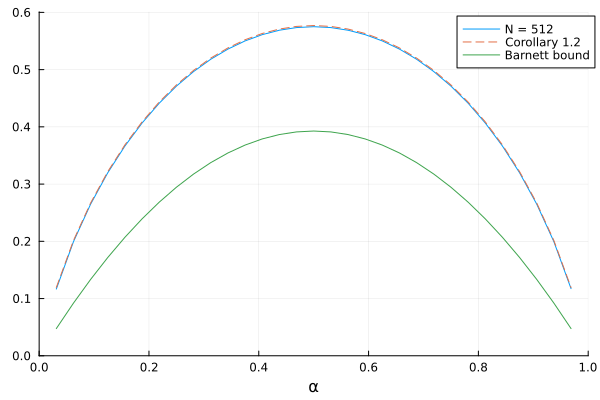}
    \caption{comparison for $\alpha = \beta$}
    \label{a7}
  \end{subfigure} \\[2 ex]
  \begin{subfigure}[t]{0.44\textwidth}
    \centering
    \includegraphics[width =\textwidth]{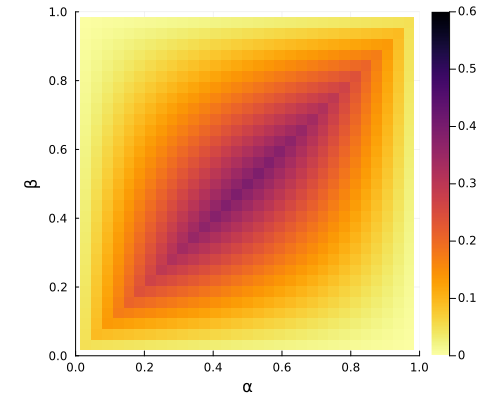}
    \caption{Barnett's lower bound}
    \label{a8}
  \end{subfigure}\hfill
  \begin{subfigure}[t]{0.53\textwidth}
    \centering
\includegraphics[width =\textwidth]{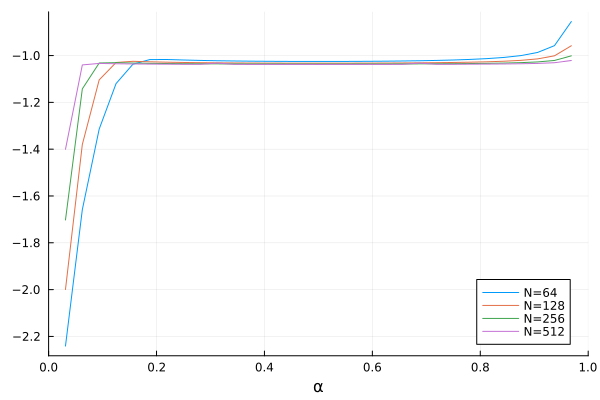}
    \caption{error term for $\log \kappa$ from \Cref{a1}}
    \label{a9}
  \end{subfigure}
 \caption{A comparison of our and Barnett's theoretical results to computed condition numbers for $N = 512$. Figure (A) contains the computed value of $\tfrac{1}{N} \log \kappa (F_{S,T})$ for contiguous $S$, $T$ with $\alpha = \tfrac{|S|}{N}$ and $\beta = \tfrac{|T|}{N}$. Barnett's lower bound $\tfrac{\pi}{2}(\min(\alpha,\beta) - \alpha \beta )$ is plotted in Figure (C), and matches Figure (A) quite well away from the diagonal $\alpha = \beta$. Figure (B) compares, for $\alpha=\beta$, $\tfrac{1}{N} \log \kappa (F_{S,T})$, Barnett's bound $\tfrac{\pi}{2} \alpha(1-\alpha)$, and the estimate $\frac{2}{\pi} \int_{0}^{\frac{\alpha \pi}{2}} \log \cot \phi\wrt\phi - \frac{1}{2 N}\log (N)$ from \Cref{a1}. Note that the estimate of \Cref{a1} is indistinguishable from the computed condition numbers. Finally, in Figure (D), we compute the unnormalized error term \\ \phantom{.}\hspace{2.6 cm} $\displaystyle{\log \kappa (F_{S,T}) -\left(\frac{2N}{\pi} \int_{0}^{\frac{\alpha \pi}{2}} \log \cot \phi\wrt\phi - \frac{1}{2}\log (N)\right)}$\\
 for $N = 64,\, 128, \, 256, \,512$. This suggests that the $O(1)$ error term in \Cref{a1} is very mild, and approximately $\approx 1.04$ for any fixed $\alpha$, as $N \rightarrow \infty$. \\[1 ex]
 All condition numbers were computed in the Julia programming language using the GenericLinearAlgebra package and the built-in BigFloat data type, with $512$ bits of precision prescribed.}
\label{a0}
\end{figure}

\begin{remark}\label{a10}
Our estimates are tight for square submatrices, but only provide loose upper bounds for non-square submatrices. However, an adaptation of the techniques presented herein can provide tighter estimates for non-square submatrices. In particular, by the Cauchy-Binet formula and singular value interlacing, for $A \in \mathbb{C}^{m \times n}$, $m > n$, we have
\begin{equation}\label{a11} \max_{|S| = n} \sigma_n\left(A_{S,[n]}\right) \le \sigma_{n}(A) \le \sqrt{\sum_{|S|=n} \sigma_n\left(A_{S,[n]}\right)^2}.
\end{equation}
Experimentally, for nearly square Fourier submatrices, the gap between the infinity and two-norm estimates in \Cref{a11} seems reasonably small. This provides a framework for bounding the nearly square case, the regime where Barnett's estimates suffer somewhat, by only using estimates for square submatrices. We leave open this potential avenue of future research.
\end{remark}

The weakness of \Cref{a3} is a high sensitivity to the minimum gap between the nodes. Consider, for example, starting with
$z_j=e^{ij\eps}$ for $j=0,\ldots,n$, the instance where \cref{a4} is tight, and replacing $z_n$ with $e^{i\eps(n-1)+i\eps^2}$. The minimum gap is now $\eps^2$, and yet \cref{a4} is still an equality, up to $O(\log n)/n$ error. It turns out that all we need for \cref{a4} to hold as an equality is for the set of nodes $\set{z_j}_{j=0}^n$ to be a ``good approximation'' to the continuous uniform measure on the arc of the unit circle from $1$ to $e^{i\eps n}$. We state this result, \Cref{a2}, in larger generality, where the continuous uniform measure on an arc of the circle is replaced by any measure satisfying a few regularity assumptions (which are satisfied, for example, when the measure is uniform on some compact set). We measure the quality of approximation a discrete set of points provides for some continuous measure via a two dimensional version of the Kolmogorov-Smirnov distance. Let $\ball(z,r)= \{w \in \mathbb{C} :\abs{w-z}\le r\}$. %
For two measures $\mu$ and $\nu$, the distance is
\[\ks(\mu,\nu)=\sup_{r>0}\sup_{z\in\C}\abs{
\mu\pare{\ball(z,r)}-\nu\pare{\ball(z,r)}
}\]
and in this context we identify a finite set of points $S \subset \mathbb{C}$ with the uniform probability measure $\mu_{S}$ on that set. 

A second way \Cref{a2} below generalizes \Cref{a3} is by extending the analysis to Vandermonde-like matrices, defined as follows. Given the polynomials $p_0,\ldots,p_n$, define
\[\pare{V^{(p_0,\ldots p_n)}(z_0,\ldots,z_n)}_{jk}=p_k(z_j).\]
The standard Vandermonde matrix is the special case of this where $p_k(z)=z^k$.
\Cref{a2} involves a couple characterizations of the polynomials, one involving the inner product $\inr fg_\nu=\int f\bar g\wrt\nu$, and another involving the sup-norm $\magn f_X=\sup_{x\in X}\abs{f(x)}$.
For the majority of reasonable choices of polynomials and nodes, \Cref{a2} shows that $\kappa(V)$ is an exponential quantity in $n$. The rate is stated in terms of a \textit{logarithmic potential}. For a measure $\mu$, its logarithmic potential is the real-valued function on $\C$
\[U^\mu(z)=\int\log\frac1{\abs{z-t}}\wrt\mu(t).\]
We present the following general theorem for Vandermonde-like matrices with arbitrary nodes in the complex plane.

\begin{theorem}\label{a2}
Given a set of nodes $S\subset\C$, $\abs S=n$, with minimum distance $\eps=\min\limits_{a,b\in S,a\neq b}\abs{a-b}$, suppose $\mu$ is a probability measure containing $S$ in its support and satisfying the following regularity assumptions for constants $\rho_1$, $\rho_2$, $C$, $\alpha$, $\beta$, and $R$:
\begin{enumerate}
    \item $\supp(\mu)$ is contained in a disk of radius $R$.
    \item $\sup_{r\in(0,\eps)}\sup_{\zeta\in\C}\frac{\mu\pare{\ball(\zeta,r)}}{r^{\beta}}\le\rho_1$
    \item $\frac{\mu(\ball(z_+,r))}{r^{\beta}}\ge\rho_2$
for $z_+=\argmax_{\zeta\in\C}U^\mu(\zeta)$ and $r=\dist(z_+,S)$.
    \item   \textit{}\(\abs{U^\mu(x)-U^\mu(y)}\le C\abs{x-y}^\alpha.\)
\end{enumerate}
Let $p_0,\ldots,p_n$ form a basis of degree $\le n$ polynomials, with Gram matrix $G=\bmat{\inr{p_j}{p_k}_\nu}_{0\le j,k\le n}$ for a probability measure $\nu$ and max norm $\gamma=\max_{0\le j\le d}\magn{p_j(z)}_{\supp(\nu)}$. Let $s=\min(\eps,\ks(\mu_S,\mu))$. Then
\spliteq{}{\frac{\log\kappa\left(V^{(p_0,\ldots,p_n)}(z_0,\ldots,z_n)\right)}n=&\sup_{z\in\supp(\mu)} U^\mu(z)-\inf_{z\in\supp(\nu)} U^\mu(z)\\+&O\pare{
\frac{\log\pare{n\gamma\magn{G^{-1}}^{1/2}}}n
+\log\frac1s\ks(\mu_S,\mu)+\ks(\mu_S,\mu)^{\alpha/\beta}}}
where $O(\cdot)$ suppresses polynomial dependencies on $\rho_1$, $\rho_2$, $C$, $\alpha$, $\beta$, and $R$.
\end{theorem}
\begin{remark}
Intuitively, this theorem says that the condition number of the Vandermonde-like matrix is an exponential quantity in $n$, with a rate given by the range of the logarithmic potential of $\mu$ between the supports of $\mu$ and $\nu$, when $\nu$ is a measure for which the polynomials $p_0,\ldots,p_n$ are orthogonal (or, at least, for which their Gram matrix is well-conditioned).
When $p_j(z)=z^j$, $\mu$ is the uniform distribution on an arc of the circle, and $\nu$ is the uniform distribution on the whole circle $\mathbb{S}^1$, \Cref{a2} recovers the rate stated in terms of the anti-derivative of $\log\cot(\cdot)$ from \Cref{a3}. However, the generality of \Cref{a2} comes at a cost, as it does not recover the correct polynomial factor from \Cref{a3}.
\end{remark}

The structure of the remainder of the paper is as follows. In \Cref{a12} we prove a number of results connecting the condition number of a Vandermonde matrix to Lagrange interpolating polynomials and show that equispaced nodes are, in some sense, extremal. In \Cref{a13} we prove tight estimates for Lagrange polynomials with equispaced nodes by viewing the logarithm of the Lagrange polynomial as a Riemann sum for the logarithmic potential of the uniform measure. Using these results, we prove \Cref{a3} and \Cref{a1,a5} in \Cref{a14}. We prove \Cref{a2} in \Cref{a15}.

\section{Vandermonde condition numbers and Lagrange interpolation}\label{a12}

A key connection exploited in this work is that the conditioning of Vandermonde and Vandermonde-like matrices can be converted to statements about the norms of Lagrange interpolating polynomials.
Recall the definition of these polynomials: given 
a set of nodes $z_0,\ldots,z_n \in \C$, there is a unique collection of polynomials,
$$L_k\pare{z;\set{z_j}_{j=0}^n}:=\prod_{j\neq k}\frac{z-z_j}{z_k-z_j},$$ abbreviated $L_k(z)$, of degree at most $n$ such that
\[L_k(z_j)=\begin{cases}1&j=k\\0&j\neq k\end{cases}.\]

First we note that estimation of the norm of a Vandermonde matrix is straightforward, allowing us to focus our attention on estimating the norm of the inverse.

\begin{lemma}\label{a16}
Let $z_0,\ldots,z_n\in\SS^1$. Then $\sqrt {n+1}\le\magn{V(z_0,\ldots,z_n)}\le n+1$. If $z_j=e^{2\pi ik_j/N}$ for distinct indices $k_j<N$, then $\sqrt{n+1}\le\magn{V(z_0,\ldots,z_n)}\le\sqrt N$ as well.
\end{lemma}
\begin{proof}
    The length of the first column of $V$ is $\sqrt{n+1}$ and the Frobenius norm of $V$ equals $n+1$. When $z_j=e^{2\pi i k_j/N}$, then $V$ is a submatrix of the $N \times N$ Fourier matrix which has norm $\sqrt N$.
\end{proof}

We supply two lemmas relating the norm of the inverse of a Vandermonde matrix to Lagrange interpolating polynomials on the same nodes. The first is sharper and is needed to recover the constant $C$ in \Cref{a3}. The second is looser, but can be more easily applied in the general setting of Vandermonde-like matrices $V^{(p_0,\ldots,p_n)}(x_0,\ldots,x_n)$ and is needed to obtain \Cref{a2}.

\begin{lemma}\label{a17}
Let $z_0,\ldots,z_n \in \mathbb{S}^1$ be distinct. Then
\[
\frac1{2\pi}\max_{k\in\set{0,\ldots,n}}\int_0^{2\pi}\abs{L_k\pare{e^{i\theta};\{z_j\}_{j =0}^n}}^2\wrt\theta
\le
\magn{V(z_0,\ldots,z_n)^{-1}}_2^2
\le\frac1{2\pi}\sum_{k=0}^n\int_0^{2\pi}\abs{L_k\pare{e^{i\theta};\{z_j\}_{j =0}^n}}^2\wrt\theta.\]
\end{lemma}
\begin{proof}
Let $V:=V(z_0,\ldots,z_n)$. Observe that
$V^{-1}\bm{e}_k$ is a list of the coefficients of $L_k(z)$, so by orthogonality of the monomials on the unit circle, we have
\[
\frac1{2\pi}\int_0^{2\pi}\abs{L_k(e^{i\theta})}^2\wrt\theta
=\magn{V^{-1} \bm{e}_k}_2^2.\]
The result follows by comparing the $\ell_2$ operator norm to the Frobenius norm,
\spliteq{\label{a18}}{
\max_{k\in\set{0,\ldots,n}} \magn{V^{-1} \bm{e}_k}_2^2
\le\magn{V^{-1}}^2_2
\le\magn{V^{-1}}_F^2
=\sum_{k=0}^n\magn{V^{-1} \bm{e}_k}_2^2.
}
\end{proof}
\begin{lemma}\label{a19}
    Let $p_0,\ldots,p_n$ be a basis of degree $\le n$ polynomials, with Gram matrix $G=[\inr{p_j}{p_k}_\nu]_{0\le j,k\le m}$ for the inner product $\inr fg_\nu=\int f\bar g\wrt\nu$.
    Put $V=V^{(p_0,\ldots,p_n)}(z_0,\ldots,z_n)$ and $L_k(z)=L_k(z,\set{z_j}_{j=0}^n)$.
    Then
    \[\dfrac{\max_{k\in\set{0,\ldots,n}}\magn{L_k}_{\supp(\nu)}}{\sqrt{n+1}\max_{0\le j\le n}\magn{p_j}_{\supp(\nu)}}
\le\magn{V^{-1}}\le
\sqrt {n+1}\magn{G^{-1}}^{1/2}\max_{k\in\set{0,\ldots,n}}\magn{L_k}_{\supp(\nu)}.\]
\end{lemma}
\begin{proof}
For $\a=\pmat{a_1&\cdots&a_n}$,
say $L_k(z)=\sum_{j=0}^na_jp_j(z)$ so that $V\a=\bm e_k$.
By the triangle inequality,
\[
\magn{L_k}_{\supp(\nu)}^2
\le
\magn{\a}_1^2
\max_{0\le j\le n}\magn{p_j}^2_{\supp(\nu)}.\]
On the other hand, \[
\magn{L_k}_{\supp(\nu)}^2
\ge\int\abs{L_k}^2\wrt\nu=a^*Ga.\]
Now $\magn{\a}_1^2\le (n+1)\magn{\a}_2^2$ and $a^*Ga\ge \magn{\a}_2^2/\rmagn{G^{-1}}$
so
\[
\frac{\magn{L_k}_{\supp(\nu)}^2}{(n+1)\max_{0\le j\le n}\magn{p_j}^2_{\supp(\nu)}}
\le\magn{\a}_2^2\le \rmagn{G^{-1}}
\magn{L_k}_{\supp(\nu)}^2
\]
Similar to \cref{a18},
\[
\max_{k\in\set{0,\ldots,n}}\magn{V^{-1}\bm e_k}^2
\le\magn{V^{-1}}^2\le
(n+1)\max_{k\in\set{0,\ldots,n}}\magn{V^{-1}\bm e_k}^2\]
gives the final result.
\end{proof}

\begin{remark}When $\supp(\nu)$ is a real interval containing the nodes, then the upper and lower bounds in \Cref{a17,a19} can be re-expressed in terms of Lebesgue constant, $\sup_{x\in\supp(\nu)}\sum_{j=0}^n\abs{L_k(x)}$.
\end{remark}

When only an upper bound is desired, as in the case of \Cref{a3}, we may reduce our analysis to the maximum value on the unit circle of a Lagrange polynomial with equally spaced points, with spacing given by the minimum gap.

\begin{lemma}\label{a20}
Let $z_0,\ldots,z_n \in \sph1$ and $\min_{j\ne k} \sdist(z_j,z_k) \ge \eps$. Then
\[ \frac{1}{n+1}   \max_{z\in\sph1} \, \abs{L_k\left(z;\{z_j\}_{j=0}^n\right)}^2 \le \frac1{2\pi}\int_0^{2\pi}\abs{L_k\left(e^{i\theta};\{z_j\}_{j =0}^n\right)}^2\wrt\theta \le \max_{z\in\sph1} \, \abs{L_k\left(z;\{z_j\}_{j=0}^n\right)}^2 \]
for all $k \in \{0,\ldots,n\}$ and
 \[\max_{k \in \{0,\ldots,n\}} \max_{z\in\sph1} \, \abs{L_k\big(z;\{z_j\}_{j=0}^n\big)} \le \max_{k \in \{0,\ldots,n\}} \max_{z\in\sph1} \,\abs{L_k\big(z;\{e^{ij \eps}\}_{j=0}^n\big)}.\]
\end{lemma}
\begin{proof}
Consider the first pair of inequalities. The integral, divided by $2 \pi$, is the average value of $\abs{L_k\big(z;\{z_j\}_{j=0}^n\big)}^2$, which is clearly bounded above by its maximum value. Let $L_k(z) = \sum_{j = 0}^n a_j z^j$. For the lower bound, we note that by Cauchy-Schwarz, for $z\in\sph1$,
\[|L_k(z)|^2 = \left|\sum_{j = 0}^n a_j z^j \right|^2 \le \left( \sum_{j=0}^n |a_j|^2 \right) \left( \sum_{j=0}^n |z|^{2j}  \right) = \frac{ (n+1) }{2\pi}\int_0^{2\pi}\abs{L_k(e^{i\theta};\{z_j\}_{j =0}^n)}^2\wrt\theta .\]
Now consider the second inequality. Suppose, without loss of generality, that $\{z_0,\ldots,z_n\}$ equals $\{e^{i\theta_0},\ldots,e^{i \theta_n}\}$ for some $0 \le \theta_0<\ldots<\theta_n \le 2 \pi$ and that the maximum is achieved by index $k$ and evaluation point $z=1$. The quantity $L_k(1)$ is non-increasing with respect to $\theta_j$ for $j>k$, and non-decreasing with respect to $\theta_j$ for $j<k$. Therefore, we may take $\theta_j = \theta_k + (j- k) \eps$ for all $j \in \{0,\ldots,n\}$ without decreasing $L_k(1)$.
\end{proof}

In a similar fashion, the determinant of a Vandermonde matrix with unit nodes may also be lower bounded by equally spaced nodes, a fact which will prove useful when analyzing nearly uniform nodes (i.e., almost $(n+1)^{th}$ roots of unity).

\begin{lemma}\label{a21}
Let $z_0,\ldots,z_n \in \mathbb{S}^1$ and $\min_{j\ne k} \sdist(z_j,z_k) \ge \eps$. Then
\[ \left|\det \left(V(z_0,\ldots,z_n) \right)\right| \ge \left|\det \left(V(1,e^{i\eps},\ldots,e^{in \eps}) \right) \right| .\]
\end{lemma}

\begin{proof}
Let $\{z_0,\ldots,z_n\}$ equal $\{e^{i\theta_0},\ldots,e^{i \theta_n}\}$ for some $0 \le \theta_0<\ldots<\theta_n \le 2 \pi$, and define $\phi_j:= \theta_j - \theta_{j-1}$ for $j = 1,\ldots,n$. We have
\begin{align*}
    \log \left|\det(\left(V(z_0,\ldots,z_n) \right)\right| &= \sum_{0\le j < k \le n} \log \left|z_j - z_k \right| \\
    &= \sum_{0\le j < k \le n}  \log \left( 2 \sin \frac{\theta_k - \theta_j}{2}  \right)\\
    &= \sum_{0\le j < k \le n}  \log \left( 2 \sin \left(\frac{1}{2} \sum_{\ell = j+1}^k \phi_\ell \right)  \right)=:f(\phi_1,\ldots,\phi_n).
\end{align*}
The function $\log(2 \sin (x/2))$ is concave on $(0,2 \pi)$, and so $f(\phi_1,\ldots,\phi_n)$, a sum of concave functions of linear forms, is also concave. The minimum value of $f(\phi_1,\ldots,\phi_n)$ on the polytope $\sum_{j=1}^n \phi_j \le 2 \pi - \epsilon$ and $\phi_j \ge \eps$ for $j = 1,\ldots,n$ occurs at an extreme point. Every extreme point corresponds to $n+1$ angles $\theta_j$ with gaps $\eps$ for all but one pair, with a gap of $2 \pi - n \epsilon$.
\end{proof}

Finally, in order to accurately estimate $\|V^{-1}\|_2$ for equispaced points on an arc, we require a better lower bound than that of \Cref{a17}. The below lemma provides the necessary test vector.

\begin{lemma}\label{a22}
Let $\eps < 2 \pi n - 2\epsilon^{1/2}$, $S = \set{j \in \{0,\ldots,n\} : \abs{j - \floor{n/2}} \le \sqrt{n} }$, and $\bm{x} \in \C^{n+1}$ be such that
\[\bm{x}_j = \begin{cases} (-1)^{n-j} e^{ij n \epsilon/2}|S|^{-1/2} & j \in S  \\ 0 & j \not\in S \end{cases}.\]
Then $\|\bm{x} \|_2 = 1$ and 
\[\|V^{-1}(1,e^{i\eps},\ldots,e^{i n  \eps}) \bm{x}\|_2^2 \ge \frac{C (1-(\xi-1)|S|)|S|}{\xi} \|V^{-1}(1,e^{i\eps},\ldots,e^{i n  \eps}) \bm{e}_{\lfloor n/2 \rfloor} \|_2^2,\]
for some constant $C>0$, where
\[\xi:= \max_{k \in S} \; \frac{\displaystyle{\int_{0}^{2 \pi} |L_{k}(e^{i \theta};\{e^{nij \eps}\}_{j = 0}^n)|^2 \wrt \theta}}{\displaystyle{\int_{n \eps + \eps}^{2 \pi - \eps} |L_{k}(e^{i \theta};\{e^{nij \eps}\}_{j = 0}^n)|^2 \wrt \theta}} .\]
\end{lemma}

\begin{proof}
Let $\{L_j(z)\}_{j = 0}^n$ be the Lagrange polynomials for $\{1,e^{i \eps},\ldots,e^{i n \eps} \}$ and $\tilde L_j(z) = (-1)^{n-j} e^{ij n \epsilon/2} L_j(z)$ for $j \in \{0,\ldots,n\}$ be a phase adjusted version whose phases match those of the corresponding entries of $\bm{x}$. We have
\[\tilde L_j(e^{i \theta}) = (-1)^{n-j} e^{ij n \epsilon/2} \prod_{k \ne j} \frac{e^{i \theta} - e^{i k \eps}}{e^{ij \eps} - e^{i k \eps}} = (-1)^{n-j} e^{\frac{i n \theta}{2}} \prod_{k \ne j} \frac{\sin \frac{\theta - k \eps}{2} }{\sin \frac{(j-k)\eps}{2}} = e^{\frac{i n \theta}{2}}\prod_{k \ne j} \frac{\sin \frac{\theta - k \eps}{2} }{\left|\sin \frac{(j-k)\eps}{2}\right|} .\]
Let 
\[\ell_j(\theta) := \prod_{k \ne j} \frac{\sin \frac{\theta - k \eps}{2} }{\left|\sin \frac{(j-k)\eps}{2}\right|}\]
for $j \in \{0,\ldots,n\}$. Note that
\begin{align*}
    \|V^{-1} \bm{x}\|_2^2 &= \frac{1}{|S|} \sum_{j,k \in S} \langle (-1)^{n-j} e^{ij n \epsilon/2}  V^{-1} \bm{e}_j, (-1)^{n-k} e^{ik n \epsilon/2}  V^{-1} \bm{e}_k \rangle \\
    &= \frac{1}{2 \pi |S|} \sum_{j,k \in S} \int_0^{2 \pi} \tilde L_j(e^{i \theta}) \overline{ \tilde L_k(e^{i \theta})} \wrt \theta \\
    &= \frac{1}{2 \pi |S|} \sum_{j,k \in S} \int_0^{2 \pi} \ell_j(\theta) \, \ell_k(\theta) \wrt \theta. 
\end{align*}
Next, we aim to upper bound the contribution from integrating over $I := [0,\eps n + \eps] \cup [2 \pi - \eps, 2 \pi]$, allowing us to focus on $[n \eps + \eps, 2 \pi - \eps]$, a region where $\ell_j(\theta) >0$ for all $j \in S$. We have
\begin{align*}
\sum_{j,k \in S} \left| \int_I \ell_j(\theta) \, \ell_k(\theta) \wrt \theta \right| &\le \sum_{j,k \in S} \left( \int_I \ell_j(\theta)^2  \wrt \theta  \int_I \ell_k(\theta)^2  \wrt \theta\right)^{1/2} \\
&\le (\xi - 1) \sum_{j,k \in S} \left( \int_{n \eps + \eps}^{2 \pi - \eps} \ell_j(\theta)^2  \wrt \theta  \int_{n \eps + \eps}^{2 \pi - \eps} \ell_k(\theta)^2  \wrt \theta\right)^{1/2} \\
&= (\xi - 1) \left( \sum_{j \in S} \left( \int_{n \eps + \eps}^{2 \pi - \eps} \ell_j(\theta)^2  \wrt \theta  \right)^{1/2}  \right)^2 \\
&\le (\xi - 1) |S| \sum_{j \in S} \int_{n \eps + \eps}^{2 \pi - \eps} \ell_j(\theta)^2  \wrt \theta,
\end{align*}
giving the lower bound
\[ \|V^{-1} \bm{x}\|_2^2 \ge \frac{1-(\xi-1)|S|}{2 \pi |S|} \sum_{j,k \in S} \int_{n \eps + \eps}^{2 \pi -\eps} \ell_j(\theta) \, \ell_k(\theta) \wrt \theta.\]
In addition, for a fixed $\theta \in [n \eps + \eps, 2 \pi - \eps]$ the value of $\ell_j(\theta)$ for $j \in S$ are all bounded below by a constant factor times $\ell_{\left\lfloor n/2\right\rfloor}(\theta)$. Note that $\ell_j(\theta)$, for $j \in S$, and $\ell_{\left\lfloor n/2\right\rfloor}(\theta)$ are both the product of sine terms and only disagree in roughly $|S|$ many of them. Suppose, without loss of generality, that $j < \lfloor n/2 \rfloor$. Then, for $\theta \in [n \eps + \eps, 2 \pi - \eps]$ and $j \in S$,
\[\left|\frac{\ell_j(\theta)}{\ell_{\left\lfloor n/2\right\rfloor}(\theta)} \right| = \left|\frac{\sin \frac{\theta - \lfloor n/2 \rfloor \eps}{2} }{\sin \frac{\theta - j \eps}{2} } \right| \prod_{m = 1}^{\lfloor n/2 \rfloor - j} \left|\frac{\sin \frac{(j+m)\eps}{2}}{\sin \frac{(n - \lfloor n/2 \rfloor + m)\eps}{2}}   \right| \ge \left( \frac{\sin \frac{j\eps}{2}}{\sin \frac{(n +1 - j)\eps}{2}} \right)^{\sqrt{n}+1}.\]
Taking logarithms of both sides and using the bound
\[\frac{\log \sin \frac{y}{2} - \log \sin \frac{x}{2}}{y-x} \le \frac{1}{2} \cot \frac{x}{2} \le \frac{1}{x} \qquad \text{for} \quad 0 <x<y< 2 \pi,\]
we obtain the lower bound
\begin{align*}
\min_{\substack{m \le \sqrt{n}\\ \theta \in [n \eps + \eps, 2 \pi - \eps]}} \log \left| \frac{\ell_{\lfloor n/2 \rfloor - m}(\theta)}{\ell_{\lfloor n/2 \rfloor }(\theta)} \right| &\ge -(\sqrt{n}+1) \left(\log \sin \left( \frac{(\frac{n}{2} + \sqrt{n} + 2)\eps}{2} \right) - \log \sin \left(\frac{(\frac{n}{2} - \sqrt{n} -1)\eps}{2}\right) \right) \\
&\ge - \frac{(\sqrt{n}+1)(2 \sqrt{n} +3)\eps}{2} \frac{2}{(\frac{n}{2} - \sqrt{n} -1)\eps}\\
&\ge C
\end{align*}
for some fixed constant $C$. Combining this with our previous lower bound, we obtain our desired result
\begin{align*}
    \|V^{-1}\bm{x}\|_2^2 &\ge \frac{e^{2C}(1-(\xi-1)|S|)}{2 \pi |S|} \sum_{j,k \in S} \int_{n \eps + \eps}^{2 \pi - \eps} \ell_{\lfloor n/2 \rfloor }(\theta)^2 \wrt \theta \\
    &=\frac{e^{2C}(1-(\xi-1)|S|)|S|}{2 \pi } \int_{n \eps + \eps}^{2 \pi - \eps} \ell_{\lfloor n/2 \rfloor }(\theta)^2 \wrt \theta \\
    &\ge \frac{e^{2C}(1-(\xi-1)|S|)|S|}{2 \pi \xi}  \int_{0}^{2 \pi} \ell_{\lfloor n/2 \rfloor }(\theta)^2 \wrt \theta \\
    &= \frac{e^{2C}(1-(\xi-1)|S|)|S|}{ \xi}  \|V^{-1} \bm{e}_{\lfloor n/2 \rfloor}\|_2^2.
\end{align*} 

\end{proof}

\section{Equispaced Lagrange interpolation and Riemann summation}\label{a13}

Here we approximate $\abs{L_k(z;\{e^{ij\eps}\}_{j=0}^n)}$ using Riemann summation. We handle the numerator and denominator separately by taking the logarithm:
\[
\log\abs{L_k\pare{z;\set{e^{ij\eps}}_{j=0}^n}}
=\sum_{j\neq k}\log\frac1{\abs{e^{ik\eps}-e^{ij\eps}}}
-\sum_{j\neq k}\log\frac1{\abs{z-e^{ij\eps}}}.\]
Define the functions
\eq{\label{a23}
U(\theta)=\sum_{j=0}^n\log\frac1{\abs{e^{i\theta}-e^{ij\eps}}},
\quad
U_k(\theta)=\sum_{j\neq k}\log\frac1{\abs{e^{i\theta}-e^{ij\eps}}}=U(\theta)+\log\abs{e^{i\theta}-e^{ik\eps}},}
so that
\(\log\abs{L_k(e^{i\theta})}=U_k(k\eps)-U_k(\theta).\) We refer the reader to \cite[Chapter 5]{b34} and \cite{b35} for an introduction to the connection between polynomial interpolation and logarithmic potential theory.
For each fixed $\theta$, the value of these functions look quite similar to Riemann sums corresponding to the anti-derivative of the function
\[
f(x)
=-\log{\abs{1-e^{ix}}}
=-\log\abs{2\sin{\frac{x}2}}.\]
The definite integral of $f(x)$ from $0$ to $\theta$ is known as the Clausen function,
\[\Cl(\theta)=-\int_0^\theta \log\abs{2\sin\frac x2}\wrt x.\]
See \Cref{a24} for a plot of $f(x)$ and $\Cl (\theta)$. We recall the following standard results regarding $f(x)$ and $\Cl(\theta)$ which will prove useful for our analysis. For brevity, in this section we use the notation $a\wedge b:=\min(a,b)$.

\begin{proposition}\label{a25}
Let $x \in (0,2 \pi)$. Then
\begin{enumerate}
\item $f'(x) = \tfrac{1}{2} \cot \tfrac{x}{2}$, $f''(x) = - \tfrac{1}{4} \csc^2 \tfrac{x}{2}$, and $f(x)$ is decreasing for $x \in (0,\pi]$, \\[-2 ex]
\item $f^{(2 \ell)}(x)>0$ for all $\ell >0$,\\[-2 ex]
\item $|f(x)|\le \log 2$ for $x \in \left(2 \arcsin \tfrac{1}{4},2 \pi - 2 \arcsin \tfrac{1}{4}\right)$,\\[-2 ex]
\item $\Cl(2 \pi - x) = - \Cl(x)$,\\[-2 ex]
\item $\Cl(x)-2\Cl(\pi+x/2)=2\Cl(x/2)$,\\[-2 ex]
\item For $x \in (0,\pi -\eps]$,
\[ \left| \frac{\Cl(x+\eps) - \Cl(x)}{\eps} - \log  x \right| \le \frac{1}{2} \frac{\eps}{x} + \log \left( \frac{2}{\pi} \right)\]
and for $x \in [\pi + \epsilon, 2 \pi)$,
\[ \left| \frac{\Cl(x) - \Cl(x-\eps)}{\eps} + \log \left(2 \pi - x \right) \right| \le \frac{1}{2} \frac{\eps}{2 \pi - x} + \log \left( \frac{2}{\pi} \right),\]\\[-2 ex]
\item For $\alpha \in (0,1)$,
\[\frac{\Cl(\alpha \pi)}{\alpha (1-\alpha) \log \tfrac{2}{\alpha}} \in \left[\tfrac{1}{\pi},\tfrac{2}{5} \right],\]\\[-2 ex]
\item $\Cl(x/2) - \Cl(\pi+x/2) = 2 \displaystyle{\int_0^{\frac{x}{4}} \log \cot \phi \wrt \phi}$.
\end{enumerate}
\end{proposition}

The remainder of the section is as follows. In \Cref{a26}, we represent the functions $U_k(\cdot)$ and $U(\cdot)$ as Riemann sums for the Clausen function. This consists of a series of lemmas, concluding with estimates for $U_k(k \epsilon)$ for $k = 0,\ldots,n$ (\Cref{a27}) and estimates for $U(\epsilon n + \pi)$ and $U(\epsilon n + \epsilon)$ (\Cref{a28}). These estimates are sufficient for estimating $\max_{k \in \{0,\ldots,n\}} \max_{z\in\sph1} \,\abs{L_k\big(z;\{e^{ij \eps}\}_{j=0}^n\big)}$, but additional analysis is required to correctly estimate the integral of $U(\theta)$. To this end, in \Cref{a29}, we produce two quadratic functions that upper and lower bound $U(\theta)$ on the interval $(n \epsilon, 2 \pi)$ (\Cref{a30}) and argue that their second derivatives are within a constant factor of each other (\Cref{a31}). In \Cref{a32}, we use these upper and lower bounds to produce tight estimates for the integral of $\abs{L_k\big(z;\{e^{ij \eps}\}_{j=0}^n\big)}^2$ (\Cref{a33}) and the sum of such integrals (\Cref{a34}).

\begin{figure}[t]
\centering
  \begin{subfigure}[t]{0.48\textwidth}
    \centering
    \includegraphics[width =.8\textwidth]{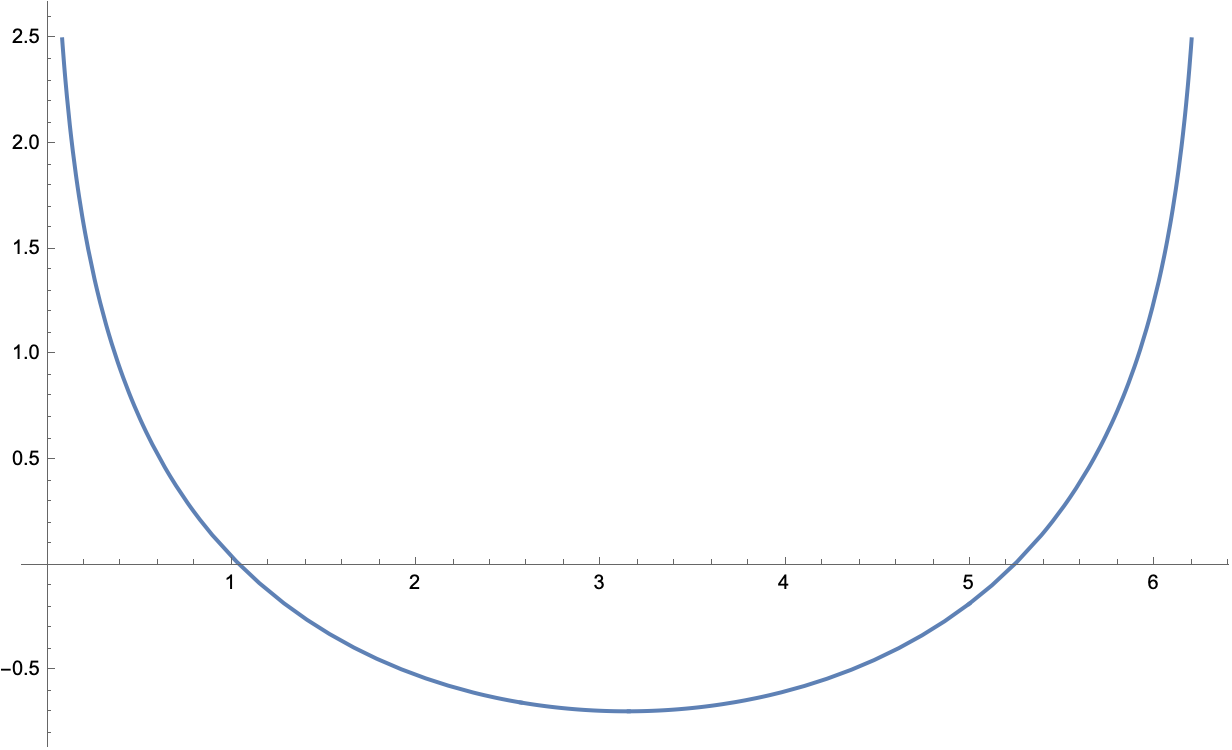}
    \caption{$f(x) = -\log\abs{2\sin\frac x2}$}
    \label{a35}
  \end{subfigure}\hfill
  \begin{subfigure}[t]{0.48\textwidth}
    \centering
\includegraphics[width =.8\textwidth]{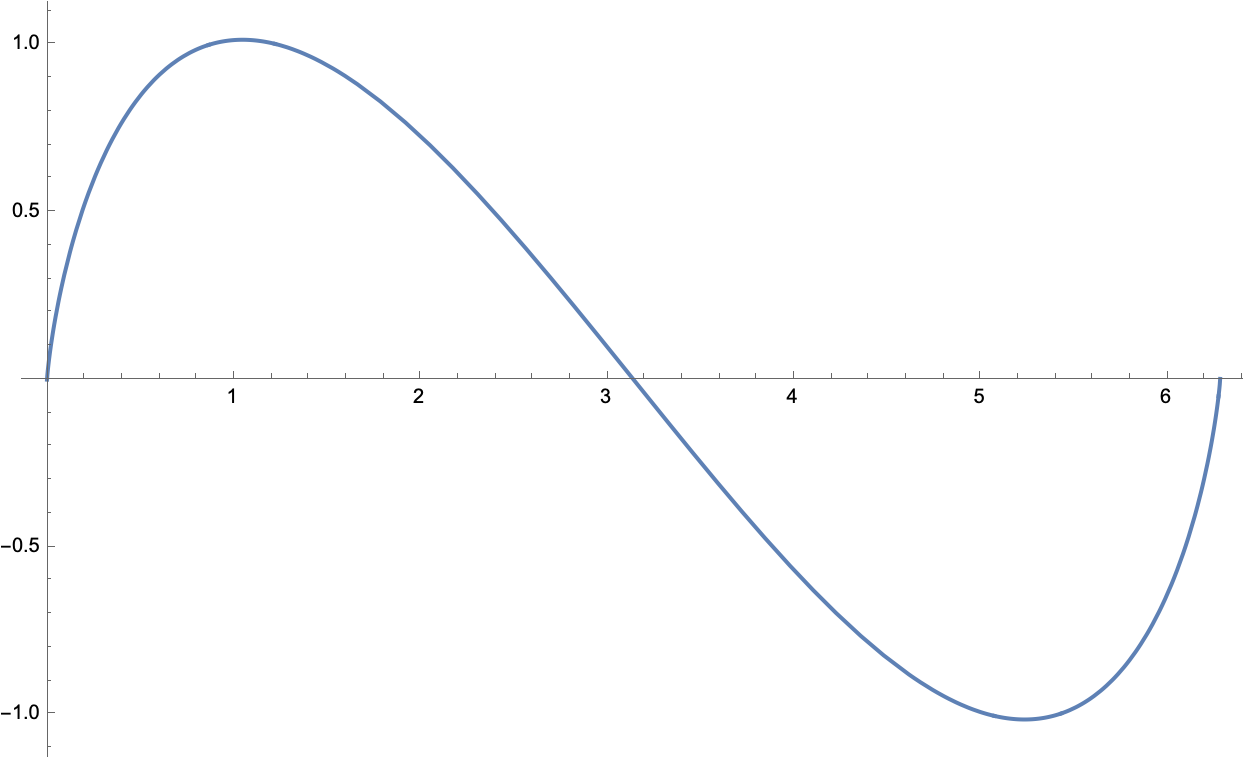}
    \caption{$\Cl(\theta) =-\int_0^\theta \log\abs{2\sin\frac x2}\wrt x$}
    \label{a36}
  \end{subfigure}
 \caption{Plots of $f(x)$ and $\Cl(\theta)$ on the interval $[0,2 \pi]$.}
\label{a24}
\end{figure}

\subsection{$U(\cdot)$ as a Riemann sum for $\Cl(\cdot)$}\label{a26}

\begin{lemma}\label{a37}
Let $(a,b)\subset[0,\pi]$.
\[
\abs{
\sum_{j=1}^m\log{\abs{2\sin\pare{\frac{a+\frac{b-a}m j}2}}}
+\frac{\Cl(b)-\Cl(a)}{\frac{b-a}m}
-\frac12\log\pare{\frac{b+\frac{b-a}m}{a+\frac{b-a}m}}
}\le \frac{3}{2}.
\]
\end{lemma}
\begin{proof}
Set $\eps=\frac{b-a}m$. Let
\spliteq{}{
A:=-\int_a^{a+\eps m}\log\abs{2\sin\frac x2}\wrt x
+
\eps\sum_{j=1}^m\log\abs{2\sin\frac{a+j\eps}2}
&=
-\sum_{j=1}^m
\int_{(j-1)\eps}^{j\eps}\log\abs{\frac{\sin\frac{a+x}2}{\sin\frac{a+j\eps}2}}\wrt x.
}
We consider the integral in the summand for each $j$. Note that the function \[f_j(x):=-\log\abs{\frac{\sin\frac{a+x}2}{\sin\frac{a+j\eps}2}}\] is decreasing and convex for $x \le \pi$, so the integral is lower bounded the axis-aligned right triangle with width $\eps$ and slope given by the slope of $f_j(x)$ at $x=j\eps$,
\eq{
\int_{(j-1)\eps}^{j\eps}f_j(x)\wrt x
\ge
\frac{\eps^2}2\abs{f'_j(j\eps)}
=
\frac{\eps^2}4\cot\frac{a+j\eps}2.}
It is similarly upper bounded (for $j\ge2$) by the triangle of width $\eps$ and slope given by the slope of $f_j(x)$ at $x=(j-1)\eps$,
\eq{
\int_{(j-1)\eps}^{j\eps}f_j(x)\wrt x
\le
\frac{\eps^2}2\abs{f'_j((j-1)\eps)}=\frac{\eps^2}4\cot\frac{a+(j-1)\eps}2.}
Computing the sum over $j$ gives
\spliteq{}{
\frac{\eps^2}4\sum_{j=1}^m\cot\frac{a+j\eps}2
\le A
  &\le\int_0^\eps f_1(x)\wrt x+\frac{\eps^2}4\sum_{j=2}^{m}\cot\frac{a+(j-1)\eps}2
\\&\le\int_0^\eps f_1(x)\wrt x
+\frac{\eps^2}4\sum_{j=1}^m\cot\frac{a+j\eps}2.}
Notably, the same sum appears on both the left and right hand side. To approximate this sum, we use $1/x - 2/\pi \le\cot x\le 1/x$ for $0 < x\le\pi/2$, giving
\[
-\frac{2m}{\pi}+\sum_{j=1}^m\frac2{a+j\eps}
\le
\sum_{j=1}^m\cot\frac{a+j\eps}2
\le
\sum_{j=1}^m\frac2{a+j\eps}
.\]
This sum can be expressed exactly in terms of the digamma function $\psi$,
\[
\sum_{j=1}^m\frac2{a+j\eps}
=\frac2\eps
\sum_{j=1}^m\frac1{\frac a\eps+j}
=\frac2\eps\pare{\psi\pare{\frac a\eps+m+1}-\psi\pare{\frac a\eps+1}},\]
which in turn admits a close approximation, \(\log x-\frac1x\le\psi(x)\le\log x-\frac1{2x}\) for $x>0$.
The remaining term to be controlled in the upper bound on $A$ is
\[
g(\eps):=\int_0^\eps f_1(x)\wrt x
=
-\int_0^\eps\log\abs{\frac{\sin\frac{a+x}2}{\sin\frac{a+\eps}2}}\wrt x.\]
Note that $\lim_{\eps\to0}g(\eps)=0$ and
\(g'(\eps)=\frac\eps2\cot\pare{\frac{a+\eps}2}\le1\), so $g(\eps)\le\eps$.
The upper bound on $A$ is therefore
\[
\eps+\frac{\eps^2}4\frac2\eps\pare{\psi\pare{\frac a\eps+m+1}-\psi\pare{\frac a\eps+1}}
\le
\eps+\frac{\eps}2\pare{\log\pare{\frac{a+(m+1)\eps}{a+\eps}}+1}
\]
and the lower bound is
\[
\frac{\eps^2}4\pare{-\frac{2m}{\pi} +\frac2\eps\pare{\psi\pare{\frac a\eps+m+1}-\psi\pare{\frac a\eps+1}}}
\ge
-\frac{m\eps^2}{2 \pi}+\frac{\eps}2\pare{\log\pare{\frac{a+(m+1)\eps}{a+\eps}}-\frac12}.
\]
Using $m\eps\le\pi$ gives the final result.
\end{proof}

\begin{lemma}\label{a38}
Let $(a,b)\subset[0,2\pi)$ be such that $\pi\in(a,b)$. Then
\[
\abs{\sum_{j=1}^m\log\abs{2\sin\frac{a+\frac{b-a}{m+1} j}2}+\frac{\Cl(b)-\Cl(a)}{\frac{b-a}{m+1}}-\frac12\log\pare{\frac\pi{a+\frac{b-a}{m+1}}}
-\frac12\log\pare{\frac{\pi}{2\pi-b+\frac{b-a}{m+1}}}
} \le 5.\]
\end{lemma}
\begin{proof}
Let $m \ge 2$ (the case $m = 1$ follows by inspection), $\eps=\frac{b-a}{m+1}$, and $k=\floor{\frac{\pi-a}\eps}$. We may split our sum into two parts:
    \[
    -\sum_{j=1}^m\log\abs{2\sin\frac{a+\eps j}2}
    =
    -\sum_{j=1}^k\log\abs{2\sin\frac{a+\eps j}2}
    -
    \sum_{j=1}^{m-k}\log\abs{2\sin\frac{2\pi-a-\eps(m+1)+\eps j}2}.\]
\Cref{a37} can be applied to each of these terms. The first term is bounded by
\[
\frac{\Cl(a+k\eps)-\Cl(a)}\eps-\frac12\log\pare{\frac{a+(k+1)\eps}{a+\eps}}\pm \frac{3}{2}
\]
and the second term is bounded by
\[
\frac{\Cl(2\pi-a-\eps(k+1))-\Cl(2\pi-a-\eps(m+1))}\eps-\frac12\log\pare{\frac{2\pi-a-\eps k}{2\pi-a-\eps m}}\pm \frac{3}{2}.\]
Furthermore,
note that $\Cl$ is odd and $2\pi$-periodic so $\Cl(2\pi-x)=-\Cl(x)$.
Then since $\pi \in [a + k \eps ,a + (k+1)\eps)$ and $|\Cl'(\theta)| \le \log 2$ for $\theta \in (\pi - \eps, \pi + \eps) \subset (2 \arcsin \tfrac{1}{4}, 2 \pi - 2 \arcsin \tfrac{1}{4})$,
\[
\abs{\frac{\Cl(a+\eps k)+\Cl(2\pi-a-\eps(k+1))}\eps}
=
\abs{\frac{\Cl(a+\eps k)-\Cl(a+\eps(k+1))}\eps}\le\log2.
\]
Additionally, $\log(a+(k+1)\eps)$ and $\log(2\pi-a-\eps k)$ both lie in $[\log\pi,\log(\pi+\eps)]$,
and so
\[ \left| \log \left(\frac{a+(k+1)\eps}{\pi} \right)\right|, \left| \log \left(\frac{2\pi-a-\eps k}{\pi} \right)\right| \le \frac{\eps}{\pi} . \]
\end{proof}

\begin{lemma}\label{a39}
Let $\epsilon >0$ and $(m+1) \epsilon < 2 \pi$. Then
\[
-\sum_{j=1}^m\log\pare{2\sin\frac{\eps j}2}
=
\frac{\Cl(\eps m)}\eps
+\frac12\log\eps
-\frac12\log\pare{\eps m\wedge(2\pi-\eps (m+1))} + O(1).
\]
\end{lemma}
\begin{proof}
    If $\eps m\le\pi$, then we may apply
\Cref{a37} to obtain
\spliteq{}{
-\sum_{j=1}^m\log\pare{2\sin\frac{\eps j}2}
  &=\frac{\Cl(\eps m)}\eps-\frac12\log\pare{m+1}+O(1)
\\&=\frac{\Cl(\eps m)}\eps-\frac12\log\pare{\eps m}+\frac12\log\eps+O(1).
}
If instead $\eps m>\pi$, we may apply \Cref{a38} along with the estimate
\[\frac{\Cl(\eps m+\eps)-\Cl(\eps m)}\eps = -\log(2\pi-(m+1)\eps )+O(1)\] in that regime to obtain
\spliteq{}{
-\sum_{j=1}^m\log\pare{2\sin\frac{\eps j}2}
  &=\frac{\Cl(\eps m+\eps)}\eps-\frac12\log\pare{\frac\pi\eps}-\frac12\log\pare{\frac\pi{2\pi-\eps m}}+O(1)
\\&=\frac{\Cl(\eps m)}\eps+\frac12\log\eps-\frac12\log\pare{2\pi-(m+1)\eps}+O(1)
.}
Combining these two inequalities gives the desired result.
\end{proof}

\begin{lemma}\label{a27}
Let $\eps >0$ and $(n+1)\eps < 2 \pi$. Then
\[U_0(0)=U_n(\eps n)=-\sum_{j=1}^n\log\abs{2\sin\frac{\eps j}2}
=\frac{\Cl(\eps n)}\eps+\frac12\log\eps-\frac12\log\pare{\eps n\wedge(2\pi-\eps (n+1))}+O(1)\]
and, for $1\le k\le n-1$,
\spliteq{}{
U_k(\eps k)
&=
\frac{\Cl(\eps k)+\Cl(\eps (n-k))}\eps
+\log\eps
\\&-\frac{\log\pare{\eps k\wedge(2\pi-\eps (k+1))}+\log\pare{\eps(n-k)\wedge(2\pi-\eps (n-k+1))}}2 + O(1).
}
\end{lemma}
\begin{proof}
This is a direct application of \Cref{a39} for $k=0,n$ and, for other $k$, an application to each sum of
\eq{\label{a40}
U_k(\eps k)=-\sum_{j\neq k}
\log\abs{2\sin\pare{\frac{\eps j-\eps k}2}}
=
-\sum_{j=1}^{n-k}
\log\abs{2\sin\pare{\frac{\eps j}2}}
-
\sum_{j=1}^k
\log\abs{2\sin\pare{\frac{\eps j}2}}.}
\end{proof}

\begin{lemma}\label{a28}
Let $\eps >0$ and $(n+2) \eps < 2 \pi$. Then
\[U\left( \frac{\eps n}{2} + \pi \right) = \frac{2 \Cl\left( \pi + \frac{\eps (n+2)}{2}\right)}{\eps} + \log \left( 2 \pi - \eps n\right) + O(1)\]
and
\[U \left(\eps n + \eps\right) = \frac{\Cl(\eps (n+1))}{\eps} + \frac{1}{2} \log \eps - \frac{1}{2} \log \left( \eps (n+1) \wedge \left( 2 \pi - \eps (n+2) \right) \right) +O(1). \]
\end{lemma}

\begin{proof}
This is a direct application of \Cref{a38} for $U\left( \frac{\eps n}{2} + \pi \right)$ and \Cref{a39} for $U \left(\eps n + \eps\right)$. For $U\left( \frac{\eps n}{2} + \pi \right)$, we have
\begin{align*}
U\left( \frac{\eps n}{2} + \pi \right) &= - \sum_{j=0}^n \log \left| 2 \sin \left( \frac{\frac{\eps n}{2} + \pi - j \eps}{2} \right)\right| \\
&= - \sum_{k=1}^{n+1} \log \left| 2 \sin \left( \frac{\pi -\frac{\eps n}{2}-\eps + k \eps}{2} \right)\right| \\
&= \frac{\Cl(\pi+\frac{\eps(n+2)}2)-\Cl(\pi-\frac{\eps(n+2)}2)}\eps
-\log\pare{\frac{\pi}{\pi-\frac{\eps n}2}}+O(1).
\end{align*}
By the symmetry $\Cl(\theta) = \Cl(2 \pi - \theta)$, we obtain our desired result.

For $U \left(\eps n + \eps\right)$, we have
\begin{align*}
U(\eps n+\eps)
&=- \sum_{j=0}^n \log \left| 2 \sin \left( \frac{\eps n + \eps - j \eps}{2} \right)\right| 
\\&=-\sum_{k=1}^{n+1} \log \left| 2 \sin \left( \frac{k\eps}{2} \right)\right|
\\&=\frac{\Cl(\eps (n+1))}{\eps} + \frac{1}{2} \log \eps - \frac{1}{2} \log \left( \eps (n+1) \wedge \left( 2 \pi - \eps (n+2) \right) \right) +O(1).
\end{align*}
\end{proof}

\subsection{Quadratic upper and lower bounds for $U(\cdot)$}\label{a29}
\newcommand{\mwide}{M_{\textnormal w}}
\newcommand{\mnarrow}{M_{\textnormal n}}
Recall our motivation for defining $U(\cdot)$ and $U_k(\cdot)$ was because we can express $\abs{L_k(e^{i\theta})}$ as the exponential
\(\abs{L_k(e^{i\theta})}=e^{U_k(\eps k)}e^{-U_k(\theta)}\), which we need to integrate the square of. Our idea is now to estimate $U(\theta)$ by a parabola, so that integrating the exponential resembles a Gaussian integral,
\[\int_{-b}^ae^{C-Mx^2}\wrt x=e^C\sqrt{\frac\pi M}\cdot\frac{\erf(a\sqrt M)+\erf(b\sqrt M)}2\]
We will use the symbols $M_{\text{wide}}=\mwide$, $M_{\text{narrow}}=\mnarrow$ to denote ``wide'' and ``narrow'' estimates for the convexity of $U(\theta)$.

\begin{lemma}\label{a30}
Let $\eps >0$, $(n+2)\eps < 2 \pi$,
\spliteq{}{
\mwide=U''\pare{\frac{\eps n}2+\pi}
,\quad
\mnarrow=2\frac{U\pare{\eps n+\eps}-U\pare{\frac{\eps n}2+\pi}}{\pare{\eps+\frac{\eps n}2-\pi}^2},
}
and define the two parabolas
\spliteq{}{
g_\star(x)&=\frac{M_\star}2\pare{x-\frac{\eps n}2-\pi}^2+U\pare{\frac{\eps n}2+\pi}}
for $\star\in\set{\textnormal w,\textnormal n}$. The polynomial $g_{\textnormal w}(\cdot)$ is the quadratic Taylor approximation of $U(\cdot)$ centered at $\frac{\eps n}2+\pi$ and $g_{\textnormal n}(\cdot)$ is the parabola intersecting $U(\cdot)$ at $\eps n+\eps$,$\frac{\eps n}2+\pi$, and $2\pi-\eps$.
Then \(g_{\textnormal w}(x)\le U(x)\) on $x\in\pare{\eps n,2\pi}$ and \(g_{\textnormal n}(x)\ge U(x)\) on $x\in\pare{\eps n+\eps,2\pi-\eps}$.
\end{lemma}
\begin{proof}
Consider the Taylor series for $U(\cdot)$ centered at $\frac{\eps n}2+\pi$. $U(\cdot)$ is the finite sum of analytic functions so is itself analytic. By examining the locations of its poles, we see that the Taylor series converges on $(\eps n,2\pi)$. By the symmetry $U(x)=U(\eps n-x)$, all the odd terms vanish, and so we have
\eq{\label{a41}
U(x)=\sum_{\ell\ge 0}\frac{U^{(2\ell)}\pare{\frac{\eps n}2+\pi}}{(2\ell)!}\pare{x-\frac{\eps n}2-\pi}^{2\ell}.}
Note that \(f^{(2\ell)}(x)>0\) for all $\ell>0$, so all of the non-constant coefficients of \cref{a41} are positive for $x \in (\epsilon n, 2 \pi)$. Consequently, $U(x)$ is lower bounded by the quadratic Taylor polynomial, which is exactly $g_{\textnormal w}(x)$. On the other hand, for $x\in(\eps n+\eps,2\pi-\eps)$ we have
\spliteq{}{
U(x)
  &=U\pare{\frac{\eps n}2+\pi}+\sum_{\ell\ge 1}\frac{U^{(2\ell)}\pare{\frac{\eps n}2+\pi}}{(2\ell)!}
\pare{\frac{x-\frac{\eps n}2-\pi}{\eps+\frac{\eps n}2-\pi}}^{2\ell}
\pare{\eps+\frac{\eps n}2-\pi}^{2\ell}
\\&\le U\pare{\frac{\eps n}2+\pi}+\pare{\frac{x-\frac{\eps n}2-\pi}{\eps+\frac{\eps n}2-\pi}}^2\sum_{\ell\ge 1}\frac{U^{(2\ell)}\pare{\frac{\eps n}2+\pi}}{(2\ell)!}\pare{\eps+\frac{\eps n}2-\pi}^{2\ell}
\\&=   U\pare{\frac{\eps n}2+\pi}+\pare{\frac{x-\frac{\eps n}2-\pi}{\eps+\frac{\eps n}2-\pi}}^2\pare{U(\eps n+\eps)-U\pare{\frac{\eps n}2+\pi}}}
which is exactly $g_{\textnormal n}(x)$.\end{proof}

\begin{lemma}\label{a42}
Let $\eps >0$ and $n \eps < 2 \pi$. Then
\[
\frac14+\frac1\eps\tan\frac{\eps n}4
\le\mwide\le
\frac14+\frac1\eps\tan\frac{\eps n}4+\frac12\tan^2\frac{\eps n}4.
\]
\end{lemma}
\begin{proof}
The derivative of $\tan(\cdot)$ is $\sec^2(\cdot)$ so
\[
2\tan\pare{\frac{\eps n}4}
=\int_{-\eps n/4}^{\eps n/4}\sec^2\pare{t}\wrt t.\]
The integral can be upper and lower bounded by the Riemann sums
\[
2\tan\pare{\frac{\eps n}4}
\le
\frac\eps2\sum_{j=0}^n
\sec^2\pare{\frac{\eps n}4-\frac{\eps j}2}
-\frac\eps2\sec^2\pare{\frac{\eps n}4-\frac{\eps\floor{n/2}}2}
\]
and
\[
2\tan\pare{\frac{\eps n}4}
\ge
\frac\eps2\sum_{j=1}^{n-1}
\sec^2\pare{\frac{\eps n}4-\frac{\eps j}2}
+\frac\eps2\sec^2\pare{0}
\]
Now note that
\[
\mwide=U''\pare{\frac{\eps n}2+\pi}
=\sum_{j=0}^nf''\pare{\frac{\eps n}2+\pi-\eps j}
=\frac14\sum_{j=0}^n\sec^2\pare{\frac{\eps n}4-\frac{\eps j}2}.\]
Plugging this in gives\spliteq{}{
2\tan\pare{\frac{\eps n}4}
&\le2\eps U''\pare{\frac{\eps n}2+\pi}
-\frac\eps2\sec^2\pare{\frac{\eps n}4-\frac{\eps\floor{n/2}}2}
\\&\le2\eps U''\pare{\frac{\eps n}2+\pi}-\frac\eps2.
}
and
\spliteq{}{
2\tan\pare{\frac{\eps n}4}
&\ge2\eps U''\pare{\frac{\eps n}2+\pi}
-\frac\eps2\sec^2\pare{-\frac{\eps n}4}
-\frac\eps2\sec^2\pare{\frac{\eps n}4}
+\frac\eps2\sec^2\pare{0}
\\&=2\eps U''\pare{\frac{\eps n}2+\pi}
-\eps\sec^2\pare{\frac{\eps n}4}
 +\frac\eps2.}
Dividing by $2\eps$ and using $1-\sec^2(x)=-\tan^2(x)$ gives the final result
\end{proof}

\begin{lemma}\label{a31}
Let $\eps >0$, $n \eps < 2 \pi - 2 \epsilon^{1/2}$, and
$\alpha:=\frac{\eps n}{2\pi}$.
There exist constants $C_1,C_2$ such that
\[
C_1\log\frac2\alpha
\le\frac{M_{\textnormal n}}{M_{\textnormal w}}\le
C_2\log\frac2\alpha.\]
\end{lemma}

\begin{proof}
By \Cref{a28} and \Cref{a25},
\begin{align*}
    \eps \mnarrow &= \frac{2 \eps}{\pi^2(1-\alpha - \eps/\pi)^2} \bigg( \frac{\Cl(\eps(n+1)) - 2 \Cl(\pi + \frac{\eps(n+2)}{2})}{\eps} \\
    &\qquad + \frac{1}{2} \log \left( \frac{\eps}{((\eps(n+1))\wedge (2 \pi - \eps(n+2)))(2 \pi - \eps n)^2} \right) +O(1) \bigg) \\
    &= \frac{4 \Cl(\eps n /2)}{\pi^2 (1-\alpha - \eps/\pi)^2} \left( 1 + \frac{O(\eps \log (\eps \alpha (1 - \alpha - \eps/\pi))}{2 \Cl(\eps n /2)} \right).
\end{align*}
By the bound 
\[\frac{2}{\pi}\frac{x}{1-x} \le \tan \left(\frac{\pi}{2} x \right) \le \frac{\pi}{2} \frac{x}{1-x} \]
for $x \in (0,1)$ and \Cref{a42},
\[\eps \mwide = C \frac{\alpha}{1-\alpha}\]
for some constant $C \ge 2/ \pi$. In addition, by \Cref{a25}, $\Cl(\eps n/2) = \hat C \alpha(1-\alpha) \log (2/\alpha)$ for some $\hat C \in [1/\pi,2/5]$. Combining these estimates and noting that $1- \alpha \ge \eps^{1/2}/\pi$, we obtain our desired result
\begin{align*}
\frac{\mnarrow}{\mwide} = \frac{4 \hat C }{\pi^2 C} \log \left( \frac{2}{\alpha} \right) \left( \frac{1-\alpha}{1-\alpha - \eps/\pi}\right)^2 \left(1 + \frac{O(\eps \log (\eps \alpha (1 - \alpha - \eps/\pi))}{\alpha(1-\alpha) \log (2 / \alpha)} \right).
\end{align*}
\end{proof}

\subsection{Integrating a Lagrange polynomial with equispaced points}\label{a32}

Here we produce tight estimates for integrals and sums of integrals of $\abs{L_k\big(z;\{e^{ij \eps}\}_{j=0}^n\big)}^2$. Throughout \Cref{a32}, let $\eps >0$, $\alpha:=\frac{\eps n}{2\pi}$, and $\beta:=\frac{\eps k}{2\pi}$, and assume throughout that $\alpha<1-\frac1\pi\sqrt\eps$. Note that $0<\beta<\alpha<1$.

\begin{lemma}\label{a43}
Let $c_{a,b}=\sup_{x\in(a,b)}\abs{\frac12\csc\frac{x-\eps k}2}$. Then
\[
\int_a^be^{-2U_k(x)}\wrt x
\le
c_{a,b}^2
e^{-2U\pare{{\eps n}/2+\pi}}
\sqrt{\frac{\pi^2\eps}2\cdot\frac{1-\alpha}\alpha}
\]
for any interval $(a,b)\subset(\eps n,2\pi)$, and 
\[
\int_a^be^{-2U_k(x)}\wrt x
\ge
e^{-2U\pare{{\eps n}/2+\pi}}
\sqrt{\frac\eps{C\log\frac2\alpha}\cdot\frac{1-\alpha}\alpha}
\]
for any interval $(a,b)=(\eps n+\eps,2\pi-\eps)$, for some absolute constant $C$.
\end{lemma}
\begin{proof}
We start by definition of $U(x)=U_k(x)+f(x-\eps k)$,
\spliteq{}{
I_{a,b}:=\int_a^be^{-2U_k(x)}\wrt x
=
\int_a^b\abs{\frac12\csc\frac{x-\eps k}2}^2e^{-2U(x)}\wrt x.}
We obtain simple upper and lower bounds by pulling the $\csc(\cdot)$ term outside the integral,
\[
\frac14\int_a^be^{-2U(x)}\wrt x
\le
I_{a,b}
\le
c_{a,b}^2\int_a^be^{-2U(x)}\wrt x.
\]
For the upper bound, if $(a,b)\subset(\eps n,2\pi)$, then we can apply the lower bound on $U(\cdot)$ from \Cref{a30},
\spliteq{}{
\int_a^be^{-2U(x)}\wrt x
\le\int_{a}^b e^{-2g_{\textnormal w}(x)}\wrt x
\le\int_{-\infty}^\infty e^{-2g_{\textnormal w}(x)}\wrt x
=e^{-2U\pare{\frac{\eps n}2+\pi}}\sqrt{\frac\pi{\mwide}}.
}
By \Cref{a42}, we have $\frac1{\mwide}\le \frac1{\frac14+\frac1\eps\tan\frac{\eps n}4}\le\frac{\pi\eps}{2}\frac{1-\alpha}\alpha$.
For the lower bound, if $(a,b)\subset(\eps n+\eps,2\pi-\eps)$, we can apply the upper bound on $U(\cdot)$ from \Cref{a30},
\[
\int_a^b e^{-2U(x)}\wrt x\ge\int_a^be^{-2g_{\textnormal n}(x)}\wrt x=e^{-2U\pare{\frac{\eps n}2+\pi}}\sqrt{\frac\pi{\mnarrow}}\frac{\erf\pare{\sqrt{\mnarrow}\pare{b-\frac{\eps n}2-\pi}}+\erf\pare{\sqrt{\mnarrow}\pare{\frac{\eps n}2+\pi-a}}}2
\]
If $a=\eps n+\eps$ and $b=2\pi-\eps$, then
\spliteq{}{
\frac{\erf\pare{\sqrt{\mnarrow}\pare{b-\frac{\eps n}2-\pi}}}{\sqrt{\mnarrow}}
=\frac{\erf\pare{\sqrt{\mnarrow}\pare{\frac{\eps n}2+\pi-a}}}{\sqrt{\mnarrow}}
&=\frac{\erf\pare{\sqrt{\mnarrow}\pare{\pi-\frac{\eps n}2-\eps}}}{\sqrt{\mnarrow}}
\\&\ge\sqrt{\frac{\pare{1-\alpha-\frac\eps\pi}}{\pare{1-\alpha-\frac\eps\pi}\mnarrow+\frac1\pi}}.
}
By \Cref{a31} and \Cref{a42},
\spliteq{}{
\mnarrow\le C\log\frac2\alpha\mwide
&\le C\log\frac2\alpha\pare{\frac14+\frac1\eps\tan\frac{\eps n}4+\frac12\tan^2\frac{\eps n}4}.
\\&\le C\log\frac2\alpha\pare{\frac\pi{2\eps}\frac\alpha{1-\alpha}+\frac14\frac1{(1-\alpha)^2}}.
}
Plugging this in and rearranging gives
\spliteq{}{
\sqrt{\frac{\pare{1-\alpha-\frac\eps\pi}}{\pare{1-\alpha-\frac\eps\pi}\mnarrow+\frac1\pi}}
&\ge\sqrt{\frac{4(1-\alpha)\eps}{2\pi\alpha C\log\frac2\alpha+\frac\eps{1-\alpha}C\log\frac2\alpha+\frac{4\eps}\pi\frac{1-\alpha}{1-\alpha-\frac\eps\pi}}}
\\&=\sqrt{4\eps\cdot\frac{1-\alpha}\alpha}\sqrt{\frac1{\pare{2\pi+\frac\eps{\alpha(1-\alpha)}}C\log\frac2\alpha+\frac{4\eps}{\pi\alpha}\frac{1-\alpha}{1-\alpha-\frac\eps\pi}}}
.}
Recall that by assumption $1-\alpha\ge\eps$, so the final result follows.
\end{proof}

\begin{lemma}\label{a33}
Set $\gamma=(\alpha-\beta)\wedge\beta+\frac\eps{2\pi}$.
For absolute constants $C_1$, $C_2$,
\[
\int_0^{2\pi}\abs{L_k(e^{i\theta})}^2\wrt\theta
\le
e^{2U_k(\eps k)-2U\pare{{\eps n}/2+\pi}}
\frac1{\gamma^2(1-\gamma)^2}
\sqrt{\frac\eps{C_1}\cdot\frac{1-\alpha}\alpha}
\]
\[
\int_0^{2\pi}\abs{L_k(e^{i\theta})}^2\wrt\theta
\ge
e^{2U_k(\eps k)-2U\pare{{\eps n}/2+\pi}}
\sqrt{\frac\eps{C_2\log\frac2\alpha}\cdot\frac{1-\alpha}\alpha}
\]
\end{lemma}
\begin{proof}
By definition,
\spliteq{}{
\int_0^{2\pi}\abs{L_k\pare{e^{ix}}}^2\wrt x
  &=e^{2U_k(\eps k)}\int_{-\eps}^{2\pi-\eps}e^{-2U_k(x)}\wrt x
\\&=e^{2U_k(\eps k)}
\pare{
\int_{-\eps}^{\eps n+\eps}e^{-2U_k(x)}\wrt x
+
\int_{\eps n+\eps}^{2\pi-\eps}e^{-2U_k(x)}\wrt x
}.
}
The second integral is both upper and lower bounded by \Cref{a43}. The first integral is lower bounded by 0, and upper bounded by
\spliteq{}{\label{a44}
\pare{\eps n+\eps-(-\eps)}e^{2g_{\textnormal w}(\eps n+\eps)}
&=\eps(n+2)e^{-2U\pare{\frac{\eps n}2+\pi}-\mwide\pare{\eps+\frac{\eps n}2-\pi}^2}
\\&\le\eps(n+2)e^{-2U\pare{\frac{\eps n}2+\pi}-\frac1\eps\tan\frac{\eps n}4\pare{\eps+\frac{\eps n}2-\pi}^2}
\\&\le 2\pi e^{-2U\pare{\frac{\eps n}2+\pi}}\exp\pare{-\frac{\pi^3}2\frac{\alpha(1-\alpha)}\eps\pare{1-\frac\eps{\pi(1-\alpha)}}^2}.
}
By assumption $1-\alpha\ge\frac1\pi\sqrt\eps$, so the expression inside the exponential is further bounded by
\[
\frac{\pi^3}2\frac{\alpha(1-\alpha)}\eps\pare{1-\frac\eps{\pi(1-\alpha)}}^2\ge\frac4{\sqrt\eps} \wedge n,\]
resulting in the overall upper bound of
\[
\int_0^{2\pi}\abs{L_k(e^{i\theta})}^2\wrt\theta
\le
e^{2U_k(\eps k)-2U\pare{{\eps n}/2+\pi}}
\pare{c_{\eps n+\eps,2\pi-\eps}^2
\sqrt{\frac{\pi^2\eps}2\cdot\frac{1-\alpha}\alpha}+2\pi \exp\left(-\left(\tfrac{4}{\sqrt\eps} \wedge n\right)\right)
}\]
for
\[
\frac12\le c_{\eps n+\eps,2\pi-\eps}
=\sup_{x\in(\eps n+\eps,2\pi-\eps)}\abs{\frac12\csc\frac{x-\eps k}2}
=\abs{\frac12\csc\frac{((n-k)\wedge k)\eps+\eps}2}
\le\frac1{2\pi\gamma(1-\gamma)}.\]
In particular, for all choices of $k$, the first term dominates the $2\pi \exp\left(-\left(\tfrac{4}{\sqrt\eps} \wedge n\right)\right)$ term.
\end{proof}

The same analysis provides an estimate for the quantity $\xi$ from \Cref{a22}.

\begin{lemma}\label{a45}
Let $\xi$ be as in \Cref{a22}. Then $\xi \le 1 + e^{-n^C}$ for some constant $C$.
\end{lemma}

\begin{proof}
By \Cref{a43} and \Cref{a44}, we have
\[ \xi = 1 + \max_{k \in S} \frac{\displaystyle{\int_{-\eps}^{n \eps + \eps} e^{-2 U_k(\theta)} \wrt \theta}}{\displaystyle{\int_{n \eps + \eps}^{2 \pi - \eps} e^{-2 U_k(\theta)} \wrt \theta}} \le 1 + \frac{2 \pi e^{- \hat C \frac{\alpha(1-\alpha)}{\eps}}}{\sqrt{\frac{\eps (1- \alpha)}{\tilde C \alpha \log (2/\alpha)}}} \le 1 + e^{-n^C}\]
for some constant $C>0$.

\end{proof}

Using \Cref{a33}, we are now prepared to estimate the sum of integrals of Lagrange polynomials.

\begin{lemma}\label{a34} For an absolute constant $C$,
\[
\sum_{k=0}^n\int_0^{2\pi}\abs{L_k(e^{i\theta})}^2\wrt\theta
\le
\frac{C\eps^2}{(1-\alpha)^6}
\exp\pare{\frac4\eps\pare{\Cl(\pi\alpha)-\Cl(\pi(1+\alpha))} }.
\]
\end{lemma}
\begin{proof}
For simplicity, we abuse notation slightly by letting $C$ denote a sequence of absolute constants that differs in value at each use.
Set $\gamma=(\alpha-\beta)\wedge\beta+\frac\eps{2\pi}$. Rearranging \Cref{a33} gives
\spliteq{}{
\int_0^{2\pi}\abs{L_k(e^{i\theta})}^2\wrt\theta
\le
\frac{\sqrt{\frac\eps{C}\frac{1-\alpha}\alpha}}{e^{2U((\alpha+1)\pi)}}
\frac{e^{2U_k(2\pi\beta)}}{\gamma^2(1-\gamma)^2}.
}
which can be combined with the bound on $U_k(\eps k)$ from \Cref{a27},
\spliteq{}{
U_k(2\pi\beta)
&=\frac{\Cl(2\pi\beta)+\Cl(2\pi(\alpha-\beta))}\eps+\log\eps
\\&\quad-\frac{\log(\beta\wedge (1-\beta-\frac\eps{2\pi}))+\log((\alpha-\beta)\wedge(1 - \alpha+\beta-\frac\eps{2\pi}))}2+O(1)
}
for $0<k<n$, i.e. $\frac\eps{2\pi}\le\beta\le\alpha-\frac\eps{2\pi}$. This results in
\spliteq{}{
\int_0^{2\pi}\abs{L_k(e^{i\theta})}^2\wrt\theta
&\le
\frac{C\eps^{\frac52}\sqrt{\frac{1-\alpha}\alpha}}{e^{2U((\alpha+1)\pi)}}
\\&\cdot
\frac1{ \pare{\beta\wedge\pare{1-\beta-\frac\eps{2\pi}})}\pare{(\alpha-\beta)\wedge\pare{1-\alpha+\beta-\frac\eps{2\pi}})} }
\frac{e^{2\frac{\Cl(2\pi\beta)+\Cl(2\pi(\alpha-\beta))}\eps}}{\gamma^2(1-\gamma)^2}
\\&\le
\frac{C\eps^{\frac52}\sqrt{\frac{1-\alpha}\alpha}}{e^{2U((\alpha+1)\pi)}}\cdot{e^{\frac4\eps\Cl(\pi\alpha)}}
\\&\cdot
\frac1{ \pare{\beta\wedge\pare{1-\beta-\frac\eps{2\pi}})}\pare{(\alpha-\beta)\wedge\pare{1-\alpha+\beta-\frac\eps{2\pi}})} }
\frac{e^{-\frac{4\pi^2}\eps\cot\pare{\frac{\pi\alpha}2}\pare{\beta-\frac\alpha2}^2}}{\gamma^2(1-\gamma)^2}
\\&\le
\frac{C\eps^{\frac52}\sqrt{\frac{1-\alpha}\alpha}}{e^{2U((\alpha+1)\pi)}}
\cdot
\frac{e^{\frac4\eps\Cl(\pi\alpha)}}{(1-\alpha)^4}
\cdot
\frac{e^{-\frac{8\pi}\eps\frac{1-\alpha}{\alpha}\pare{\beta-\frac\alpha2}^2}}{\beta^3(\alpha-\beta)^3}
}
where we used a quadratic approximation of the dominant term in the exponent,
\[
\frac{\Cl(2\pi\beta)+\Cl(2\pi(\alpha-\beta))}\eps
\le\frac{2\Cl(\pi\alpha)-\frac12\cot(\pi\alpha/2)(2\pi\beta-\pi\alpha)^2}\eps.\]
Thus
\spliteq{}{
\sum_{k=1}^{n-1}\abs{L_k(e^{i\theta})}^2\wrt\theta
&\le
\frac{C\eps^{\frac52}\sqrt{\frac{1-\alpha}\alpha}}{e^{2U((\alpha+1)\pi)}}
\frac{e^{\frac4\eps\Cl(\pi\alpha)}}{(1-\alpha)^4}
\sum_{k=1}^{n-1}\frac{e^{-\frac{8\pi}\eps\frac{1-\alpha}\alpha\pare{\beta-\frac\alpha2}^2}}{\beta^3(\alpha-\beta)^3}.
}
When $\beta\in\pare{\frac13\alpha,\frac23\alpha}$, observe the denominator is lower bounded by $\frac8{9^3}\alpha^6$, so the sum can be approximated by the Gaussian integral. For the other $\beta$, the exponential term is upper bounded by $e^{-\frac{8\pi}\eps(1-\alpha)\alpha\frac1{36}}$, and the sum can be bounded by a $p$-series. This gives
\[
\sum_{k=1}^{n-1}\abs{L_k(e^{i\theta})}^2\wrt\theta
\le
\frac{C\eps^{\frac52}\sqrt{\frac{1-\alpha}\alpha}}{e^{2U((\alpha+1)\pi)}}
\frac{e^{\frac4\eps\Cl(\pi\alpha)}}{(1-\alpha)^4}
\pare{\frac1\eps
\sqrt{\frac{\eps\alpha}{1-\alpha}}+\frac{e^{c\alpha(1-\alpha)/\eps}}{\eps^3}}
\le
C\eps^2
\frac{e^{\frac4\eps\Cl(\pi\alpha)}}{(1-\alpha)^4}
e^{-2U((\alpha+1)\pi)}.
\]
The $k=0$ and $k = n$ (i.e., $\beta=0,\alpha$), terms are bounded via \Cref{a27}
\[e^{2U_0(0)}+e^{2U_n(\eps n)}=\frac{C\eps}{\alpha\wedge(1-\alpha)}e^{2\frac{\Cl(\eps n)}\eps},\]
so these terms are negligible in the final sum.
\end{proof}

\section{Proofs of \Cref{a3} and \Cref{a1,a5}}\label{a14}

Using the connection between Vandermonde matrices and Lagrange interpolating polynomials (\Cref{a12}) and the relationship between Lagrange polynomials for equispaced nodes and Riemann summation of logarithmic potentials (\Cref{a13}), we are prepared to prove \Cref{a3} and the resulting \Cref{a1,a5}. We break our proof of \Cref{a3} into three lemmas, depending on the value of $n \eps$.

\subsection{Gap satisfies $2 \pi - n \eps \le 2 \eps$}

When the nodes are nearly equispaced around the entire unit circle, a number of our estimates from \Cref{a13} are no longer valid, requiring alternate techniques. However, in this regime, we may use alternate techniques that exploit our matrix's closeness to a Fourier matrix.

\begin{lemma}\label{a46}
Let $z_0,\ldots,z_n \in \mathbb{S}^1$, $ \min_{j \ne k} \sdist(z_j,z_k) \ge \eps>0$, and $2 \pi - n \eps \le 2 \eps$. Then
\[\log \magn{V(z_0,\ldots,z_n)^{-1}}_2 = \frac4{\eps} \int_{0}^{\frac{\eps n}4} \log \cot\phi\wrt\phi  \pm O\pare{\log (2\pi - \eps n)}.\]
\end{lemma}

\begin{proof}
First we note that
\[\frac4{\eps} \left|\int_{0}^{\frac{\eps n}4} \log \cot\phi\wrt\phi \right| = \frac4{\eps} \left|\int_{\frac{\eps n}4}^{\frac{\pi}2} \log \cot\phi\wrt\phi \right| \le \frac4{\eps} \left|\int_{0}^{\eps /2} \log \cot\phi\wrt\phi \right|= O(\log \eps).\]
What remains is to prove that $\magn{V(z_0,\ldots,z_n)^{-1}}_2$ is bounded above by a polynomial in $n$. We note that $\eps \ge \delta:=2 \pi/(n+2)$, and so, by \Cref{a21},
\[\left|\det \left(V(z_0,\ldots,z_n) \right)\right| \ge \left|\det \left(V(1,e^{i\delta},\ldots,e^{in \delta}) \right) \right|. \]
The matrix $V(1,e^{i\delta},\ldots,e^{in \delta})$ is the leading $(n+1)\times(n+1)$ submatrix of the $(n+2)$ dimensional discrete Fourier matrix, and so
\[V(1,e^{i\delta},\ldots,e^{in \delta})^*V(1,e^{i\delta},\ldots,e^{in \delta}) = (n+2)I - \bm{v} \bm{v}^*\]
for some vector $\bm{v}$ with $\|\bm{v}\|_2 = \sqrt{n+1}$. Therefore,
\[|\det\left(V(z_0,\ldots,z_n) \right)|^2 \ge |\det(V(1,e^{i\delta},\ldots,e^{in \delta}))|^2 = (n+2)^{n+1}(n+2 - \bm{v}^* \bm{v}) = (n+2)^{n+1}.\]
In addition, 
\[\mathrm{trace}\left( V(z_0,\ldots,z_n)^* V(z_0,\ldots,z_n) \right) = (n+1)^2.\]
Therefore, by the AM-GM inequality,
\[\magn{V(z_0,\ldots,z_n)^{-1}}_2^2 =\frac{\prod_{j=1}^n \sigma_{j}^2(V(z_0,\ldots,z_n))}{\prod_{j=1}^{n+1} \sigma_{j}^2(V(z_0,\ldots,z_n))} \le \frac{\left( \frac{1}{n} \sum_{j=1}^n \sigma_{j}^2(V(z_0,\ldots,z_n)) \right)^n }{(n+2)^{n+1}} \le  \frac{ \left( \frac{(n+1)^2}{n} \right)^n }{(n+2)^{n+1}} . \]
The quantity $\left(\tfrac{(n+1)^2}{n(n+2)}\right)^n$ is at most $4/3$ for $n \in \mathbb{N}$, and so
\[  \magn{V(z_0,\ldots,z_n)^{-1}}_2^2 \le \frac{1}{n+2} \left(\frac{(n+1)^2}{n(n+2)}\right)^n \le \frac{4}{3(n+2)},\]
completing the proof.
\end{proof}

\subsection{Gap satisfies $2 \eps < 2 \pi - n \eps < 2\eps^{1/2}$}

When our matrix is at least some bounded distance away from being a Fourier matrix, but still has nearly equally spaced nodes, only half of our estimates from \Cref{a13} apply. In particular, while our estimates for the maximum of a Lagrange polynomial are valid, we do not have enough of a gap to produce estimates of integrals. However, because the gap is so small, estimates of the maximum over the unit circle suffice.

\begin{lemma}\label{a47}
Let $z_0,\ldots,z_n \in \mathbb{S}^1$, $ \min_{j \ne k} \sdist(z_j,z_k) \ge \eps>0$, and $2 \eps < 2 \pi - n \eps < 2\eps^{1/2}$. Then
\[\log \magn{V(z_0,\ldots,z_n)^{-1}}_2 \le \frac4{\eps} \int_{0}^{\frac{\eps n}4} \log \cot\phi\wrt\phi  \pm O\pare{\log (2\pi - \eps n)},\]
with equality when $z_j = e^{i j \eps}$ for $j \in \{0,\ldots,n\}$.
\end{lemma}

\begin{proof}
In the regime $2 \eps < 2 \pi - n \eps < 2\eps^{1/2}$, $\log n$ and $\log \eps$ are $O(\log(2\pi - n \eps))$, and so we need not concern ourselves with logarithmic error terms. By \Cref{a17,a20},
\begin{align*}
\log \magn{V(z_0,\ldots,z_n)^{-1}}_2 &= \frac{1}{2} \max_{k \in \{0,\ldots,n\}} \log \left( \int_0^{2\pi}\abs{L_k\big(e^{i \theta};\{z_j\}_{j=0}^n\big)}^2\wrt\theta \right) + O(\log n) \\
&= \max_{k \in \{0,\ldots,n\}} \max_{z \in \mathbb{S}^1} \log \left(\abs{L_k\big(z;\{z_j\}_{j=0}^n\big)}\right) + O(\log n) \\
&\le \max_{k \in \{0,\ldots,n\}} \max_{z \in \mathbb{S}^1} \log \left(\abs{L_k\big(z;\{e^{ij\eps}\}_{j=0}^n\big)}\right) + O(\log n). 
\end{align*}
What remains is to estimate 
\[\max_{k \in \{0,\ldots,n\}} \max_{z \in \mathbb{S}^1} \log \left(\abs{L_k\big(z;\{e^{ij\eps}\}_{j=0}^n\big)}\right) = \max_{k \in \{0,\ldots,n\}} \left(U_k(k \eps) - \min_{\theta \in [0,2 \pi)} \left(U(\theta) + \log \left| e^{i \theta} - e^{i k \eps} \right| \right) \right)\] 
exactly, up to logarithmic error. Because $2 \pi - n \eps > 2 \eps$, clearly $\max_{k \in \{0,\ldots,n\}} U_k(k \eps) = U_{\lfloor n/2 \rfloor}(\lfloor n/2 \rfloor \eps)$ and $\min_{\theta \in [0,2 \pi)} U(\theta)$ is achieved in the interval $(n \epsilon, 2 \pi)$. Furthermore, $U(\theta)$ is a sum of convex functions, so it is also convex on $(n \epsilon, 2 \pi)$. By symmetry about $n \epsilon/2 + \pi$, $\min_{\theta \in [0,2 \pi)} U(\theta) = U(n \epsilon/2 + \pi)$. 

By \Cref{a27,a28} and $2 \eps< 2 \pi - n \eps < 2 \eps^{1/2}$,
\begin{align*} U_{\lfloor n/2 \rfloor}(\lfloor n/2 \rfloor \eps) - U(n \epsilon/2 + \pi) &=\frac{\Cl(\eps \lfloor n/2 \rfloor) + \Cl(\eps \lceil n/2 \rceil)- 2\Cl (\pi + \eps(n+2)/2)}{\eps} + O(\log(2\pi - n \eps)) \\
&= \frac{2}{\epsilon} \left(\Cl(\eps n/2 ) - \Cl (\pi + \eps n/2) \right) + O(\log(2\pi - n \eps)) \\
&= \frac{4}{\eps} \int_{0}^{\frac{\eps n}4} \log \cot\phi\wrt\phi + O(\log(2\pi - n \eps)).
\end{align*}
Finally, note that the difference between $U_{\lfloor n/2 \rfloor}(\lfloor n/2 \rfloor \eps) - U(n \epsilon/2 + \pi)$ and the maximum over $k$ and $\theta$ is $O(1)$, completing the proof.
\end{proof}

\subsection{Gap satisfies $ 2 \pi - n \eps > 2 \eps^{1/2}$}

Here we consider the final, general case where $n \epsilon$ is bounded away from $2 \pi$ by $\epsilon^{1/2}$. Here we have access to all of the machinery derived in \Cref{a13}, but also require much tighter estimates.

\begin{lemma}\label{a48}
Let $z_0,\ldots,z_n \in \mathbb{S}^1$, $ \min_{j \ne k} \sdist(z_j,z_k) \ge \eps>0$, and $2 \pi - n \eps \ge 2\eps^{1/2}$. Then
 \[\log\magn{V(z_0,\ldots,z_n)^{-1}} \le \frac4{\eps} \int_{0}^{\frac{\eps n}4} \log \cot\phi\wrt\phi -\frac{1}{2} \log \frac{1}\eps \pm O(\log (n \eps(2\pi - n \eps))).\]
Furthermore,
  \[\log\magn{V(1,e^{i \eps},\ldots,e^{i n \eps})^{-1}} = \frac4{\eps} \int_{0}^{\frac{\eps n}4} \log \cot\phi\wrt\phi - \log \frac{1}\eps \pm O(\log (n \eps(2\pi - n \eps))).\]
\end{lemma}

\begin{proof}
First we consider the upper bound. By \Cref{a17,a20,a27},
\begin{align*}\log\magn{V(z_0,\ldots,z_n)^{-1}} &\le \max_{k\in\{0,\ldots,n\}} \max_{z \in \mathbb{S}^1} \, \log \abs{L_k\big(z;\{e^{ij \eps}\}_{j=0}^n\big)} + \frac{1}{2} \log n \\
&\le \max_{k\in\{0,\ldots,n\}} U_{k}(k \eps) - U(n \eps/2 + \pi) + \frac{1}{2} \log n + O(\log (2 \pi - n \eps)) \\
&= U_{\lfloor n/2 \rfloor}(\lfloor n/2 \rfloor) - U(n \eps/2 + \pi) + \frac{1}{2} \log n + O(\log (n \eps(2 \pi - n \eps))),
\end{align*}
and, by \Cref{a27,a28},
\begin{align*}
U_{\lfloor n/2 \rfloor}(\lfloor n/2 \rfloor) - U(n \eps/2 + \pi) &= \frac{2}{\eps}\left( \Cl(\eps \lfloor n/2 \rfloor) - \Cl(\pi + \eps n/2) \right) + \log \eps + O(\log (n \eps(2 \pi - n \eps))) \\
&=\frac{4}{\eps} \int_{0}^{\frac{\eps n}4} \log \cot\phi\wrt\phi + \log \eps + O(\log (n \eps(2 \pi - n \eps))),
\end{align*}
completing the proof of the upper bound. For our desired equality, \Cref{a17,a34} give the upper bound
\[\log\magn{V(1,e^{i \eps},\ldots,e^{i n \eps})^{-1}} \le \frac{2}{\eps}\left(\Cl(n \eps/2) - \Cl(\pi + n \eps/2) \right) + \log \eps + O(\log (n \eps(2\pi - n \eps))).\]
What remains is to produce a matching lower bound.
By \Cref{a45}, the quantity $\xi-1$ in \Cref{a22} is exponentially small in $n$. Therefore, by \Cref{a22,a33},
\begin{align*}\log\magn{V(1,e^{i \eps},\ldots,e^{i n \eps})^{-1}} &\ge \frac{1}{4} \log n + \frac{1}{2}\log \left(\int_0^{2\pi}\abs{L_{\lfloor n/2 \rfloor}(e^{i\theta})}^2\wrt\theta \right) +O(1) \\
&\ge \frac{1}{4} \log n + \frac{1}{4} \log \eps + U_{\lfloor n/2 \rfloor} (\eps \lfloor n/2 \rfloor) - U(\eps n/2 + \pi) + O(\log (n \eps(2\pi - n \eps)))\\
&= \frac{2}{\eps}\left(\Cl(n \eps/2) - \Cl(\pi + n \eps/2) \right) + \log \eps + O(\log (n \eps(2\pi - n \eps))),
\end{align*}
completing the proof.
\end{proof}

\subsection{Proofs of \Cref{a3} and \Cref{a1,a5}}\label{a49}

\begin{proof}[Proof of \Cref{a3}]
\Cref{a3} follows from applying \Cref{a46} when $2\pi - n \eps \le 2 \eps$, \Cref{a47} when $2\eps< 2 \pi - n \eps < 2 \eps^{1/2}$, and \Cref{a48} when $2 \pi - n \eps \ge 2 \eps^{1/2}$.
\end{proof}

\begin{proof}[Proof of \Cref{a1}]
Without loss of generality, we may assume that $|S| = |T|$ and restrict our attention to the smallest singular value, as $ \sqrt{\max(|S|,|T|)} \le \sigma_1(F_{S,T}) \le \sqrt{N}$ for all $S,T \subset[N]$ and the smallest singular value $\sigma_{\min(m , n)}(A)$ of a matrix $A \in \mathbb{R}^{m \times n}$ with $m<n$ (resp. $m>n$) is non-increasing when at most $n-m$ rows (resp. $m-n$ columns) are added. Applying \Cref{a3} with $n = \alpha N -1$ and minimum angle distance $\eps = 2 \pi /N$, and noting that
\[ \int_{\frac{\alpha \pi}{2} -\frac{\pi}{2N}}^{\frac{\alpha \pi}{2}}  \log \cot\phi\wrt\phi =\frac{O\left((\log \left( \alpha (1 - \alpha) \right) \right)}{N},\]
we obtain our desired upper bounds. Furthermore, by \Cref{a3}, equality is obtained when $S = T = \{1,\ldots,|S|\}$, or, equivalently, when $|S| = |T|$ and both are cyclically contiguous subsets.
\end{proof}

\begin{proof}[Proof of \Cref{a5}]
When $\alpha \in [1/3,2/3]$, \Cref{a5} follows immediately from \Cref{a1}. When $\alpha \not \in [1/3,2/3]$, we must show that the $O(\log (\alpha(1-\alpha)))$ term from \Cref{a1} is negligible. Indeed, we have
\[\frac{2 G N}{\pi} - \frac{2 N}{\pi}\int_{0}^{\frac{\alpha \pi}{2}} \log \cot \phi\wrt\phi \ge  \frac{2 N}{\pi}  \int_{\frac{\pi}{6}}^{\frac{\pi}{4}} \log \cot \phi\wrt\phi  > \frac{N}{25}, \]
which is asymptotically larger than $O\left( \log (\alpha(1-\alpha))\right)$, completing the proof.
\end{proof}

\section{Conditioning for any polynomials and any measures}\label{a15}

Here we prove \Cref{a2} using two lemmas. The first lemma (\Cref{a50}) bounds the difference in logarithmic potential between two measures close in KS-distance. The second lemma (\Cref{a51}), using the first, produces tight bounds for the maximum value of a Lagrange polynomial in terms of a logarithmic potential. Combining this with the previously established connection between Vandermonde-like matrices and Lagrange polynomials (\Cref{a19}) completes the proof of \Cref{a2}.

\begin{lemma}\label{a50}
Let $\mu_c$ and $\mu_d$ be two probability measures supported on a disk of diameter $R$.
Fix any $z\in\C$. Pick $\beta,\rho>0$ and
\[\eps\le\min\pare{\dist(z,\supp(\mu_d)),\pare{\frac{\beta}{\rho}\ks(\mu_d,\mu_c)}^{1/\beta}}.\]
If
\[\sup_{r\in(0,\eps)}\sup_{\zeta\in\C}\frac{\mu_c(\ball(\zeta,r))}{r^\beta}\le\rho.\]
Then
\[\abs{U^{\mu_d}(z)-U^{\mu_c}(z)}\le\log\pare{\frac{eR}\eps}\ks(\mu_d,\mu_c).\]
\end{lemma}
\begin{proof}
Let $(X)_+:=\max(X,0)$ and $\Omega$ be the interior of $\ball(z,\eps)$.
For probability measures $\mu$, we can decompose
\[
U^\mu(z)
=\int_\C\log\frac1{\abs{z-\zeta}}\wrt\mu(\zeta)
=\log\frac1\eps
+\int_\Omega\log\frac\eps{\abs{z-\zeta}}\wrt\mu(\zeta)
-\int_\C\pare{\log\frac{\abs{z-\zeta}}\eps}_+\wrt\mu(\zeta).
\]
This can be applied to $\mu_c$ and $\mu_d$ along with the triangle inequality to obtain
\spliteq{}{
    \abs{U^{\mu_d}(z)-U^{\mu_c}(z)}
    &\le
\int_\Omega\log\frac\eps{\abs{z-\zeta}}\wrt\mu_d(\zeta)
+\int_\Omega\log\frac\eps{\abs{z-\zeta}}\wrt\mu_c(\zeta)
\\&
+\abs{\int_\C\pare{\log\frac{\abs{z-\zeta}}\eps}_+\wrt\mu_c(\zeta)
-\int_\C\pare{\log\frac{\abs{z-\zeta}}\eps}_+\wrt\mu_d(\zeta)}
.}
The first term vanishes since the support of $\mu_d$ misses $\Omega$. For the second term, observe that
\spliteq{\label{a52}}{
 \int_\Omega\log\frac\eps{\abs{z-\zeta}}\wrt\mu_c(\zeta)
  &=\int_0^\infty\mu_c\pare{\set{\zeta:\log\frac\eps{\abs{z-\zeta}}\ge t}}\wrt t
\\&=\int_0^\infty\mu_c\pare{\ball(z,\eps e^{-t})}\wrt t
\\&\le\int_0^\infty\rho\eps^\beta e^{-t\beta}\wrt t
\\&={\rho\eps^\beta}/\beta
\\&\le\ks(\mu_d,\mu_c)
.}
For the third term, note that
\spliteq{}{
\int_\C\pare{\log\frac{\abs{z-\zeta}}\eps}_+\wrt\mu(\zeta)
&=
\int_0^{\log(R/\eps)}\mu\pare{\set{\zeta:\log\frac{\abs{z-\zeta}}\eps\ge t}}\wrt t
\\&=
\int_0^{\log(R/\eps)}\mu\pare{\C\backslash\ball\pare{z,\eps e^t}}\wrt t
}
which when applied to $\mu_c$ and $\mu_d$ gives a bound of
\spliteq{\label{a53}}{
&\abs{
\int_0^{\log(R/\eps)}\biggr(
\mu_d\pare{\C\backslash\ball\pare{z,\eps e^t}}
-
\mu_c\pare{\C\backslash\ball\pare{z,\eps e^t}}
\biggr)\wrt t}
\\&\le\log(R/\eps)\sup_{0\le t\le\log(R/\eps)}
\biggr|\mu_d\pare{\ball\pare{z,\eps e^t}}
-
\mu_c\pare{\ball\pare{z,\eps e^t}}\biggr|
\\&\le\log(R/\eps)\ks(\mu_d,\mu_c).}
Adding \cref{a52,a53} gives the result.
\end{proof}

\begin{lemma}\label{a51}
Fix $S=\set{z_0,\ldots,z_n}\subset\C$ and set $\eps=\min_{j\neq k}\abs{z_j-z_k}$.
Let $\mu$ be a probability measure satisfying the conditions of \Cref{a2}.
\newcommand{\range}{\textnormal{range}}
Set $\Delta=\sup_{z\in\supp(\mu)}U^\mu(z)-\inf_{z\in\supp(\nu)}U^\mu(z)$. Put $s=\min(\eps,\ks(\mu_S,\mu))$. Then
\begin{align}
\max_{0\le k\le n}\sup_{z\in\supp(\nu)}&\frac{\log L_k(z)}n
\le
\Delta+O\pare{\log\pare{1/s}\ks(\mu_S,\mu)}.\label{a54}
\\
\max_{0\le k\le n}\sup_{z\in\supp(\nu)}&\frac{\log L_k(z)}n
\ge
\Delta-O\pare{\log\pare{1/s}\ks(\mu_S,\mu)+\ks(\mu_S,\mu)^{\alpha/\beta}}
\end{align}
where the Ohs suppress dependency on $\rho_1,\rho_2,R,\beta,\alpha$.
\end{lemma}
\begin{proof}
Denote
\[\mu_k=\frac1n\sum_{j\neq k}\delta_{z_j},\quad
\mu_S=\frac1{n+1}\sum_{j=0}^n\delta_{z_j}.\]
and\[z_+=\argmax\limits_{z\in\supp(\mu)} U^\mu(z),\quad z_-=\argmin\limits_{z\in\supp(\nu)} U^\mu(z),\quad
z_-^{(k)}=\argmin\limits_{z\in\supp(\nu)}U^{\mu_k}(z)
.\]
Observe by definition that
\[\sup_{z\in\supp(\nu)}\frac{\log L_k(z)}n=U^{\mu_k}(z_k)-U^{\mu_k}(z_-^{(k)}).\]
We claim that
\eq{\label{a55}\frac12\ks(\mu_S,\mu)\le\ks(\mu_k,\mu)\le2\ks(\mu_S,\mu).}
To see this, note that condition (3) implies $\mu(\set z)=0$ for any $z$. In particular, we have $\ks(\mu_k,\mu_S)=\frac1{n+1}\le\min\pare{\ks(\mu,\mu_k),\ks(\mu,\mu_S)}$. \cref{a55} follows by the triangle inequality. Now set \[\eps'=\min\pare{\eps,(\beta\ks(\mu_S,\mu)/2\rho_1)^{1/\beta}}\]
and observe that $\eps'$ satisfies the hypothesis of \Cref{a50} for $\mu_d=\mu_k$, $\mu_c=\mu$, and $z=z_k$.
Define $\delta(r)=2\log\pare{\frac{eR}r}\ks(\mu_S,\mu)$ so that \Cref{a50} along with \cref{a55} and the definition of $z_\pm$ implies
\spliteq{}{
U^{\mu_k}(z_k)
\le U^\mu(z_k)+\log\pare{\frac{eR}{\eps'}}\ks(\mu_k,\mu)
\le U^\mu(z_k)+\delta(\eps')
\le U^\mu(z_+)+\delta(\eps')}
and also
\spliteq{}{
U^{\mu_k}(z_-^{(k)})
\ge U^\mu(z_-^{(k)})-\log\pare{\frac{eR}{\eps'}}\ks(\mu_k,\mu)
\ge U^\mu(z_-^{(k)})-\delta(\eps')
\ge U^\mu(z_-)-\delta(\eps').}
Noting that $\log(R/{\eps'})=O\pare{\log\frac1{\min(\eps,\ks(\mu_S,\mu))}}$ gives the first inequality \cref{a54}.
For second inequality, we start by again applying \Cref{a50},
\spliteq{}{
\max_{0\le k\le n} U^{\mu_k}(z_k)
  &\ge\max_{0\le k\le n}U^\mu(z_k)-\delta(\eps')
\\&\ge U^\mu(z_+)-\delta(\eps')-C\min_{0\le k\le n}\abs{z_k-z_+}^\alpha
\\&= U^\mu(z_+)-\delta(\eps')-C\dist(z_+,S)^\alpha.
}
Notice that
\(\mu_S\pare{ \ball(z_+,\dist(z_+,S)) }=0\)
so
\[\ks(\mu,\mu_S)\ge\mu\pare{ \ball(z_+,\dist(z_+,S)) }\ge\rho_2\dist(z_+,S)^\beta.\]
Applying \Cref{a50} one last time,
\spliteq{}{
U^{\mu_k}(z_-^{(k)})
\le U^{\mu_k}(a)
  &\le U^\mu(a)+\delta(\dist(a,S))
\\&\le U^\mu(z_-)+\delta(\dist(a,S))+C\abs{a-z_-}^\alpha
.}
Our final step is to properly select $a$. We pick $a$ to be any element of the set
\[\ball(z_-,(n+1)^{1/2}r)\backslash\bigcup_{j=0}^n\ball(z,r)\]
for $r=\frac1{2R}(\eps')^2$, which is nonempty by considering the areas of each ball.
Since $S$ is contained in a disk of radius $R$, the points cannot all be pairwise for apart. In particular, by area considerations we must have $(n+1)^{1/2}\eps<2R$.
Consequently,
\[\abs{a-z_-}\le(n+1)^{1/2}r<\eps'\le(\beta\ks(\mu_S,\mu)/\rho_1)^{1/\beta}.\]
Additionally, by construction $\dist(a,S)\ge r$ and $\delta(r)=O\pare{ \log(R/\eps')\ks(\mu_S,\mu) }$, so we have the final result.
\end{proof}

\begin{proof}[Proof of \Cref{a2}]
This is a direct combination of \Cref{a19} and \Cref{a51}.
\end{proof}

{}

\section*{Acknowledgments}
This material is based upon work supported by the National Science Foundation under grant no. DMS-2513687. The second author thanks Alex Barnett for introducing them to the problem. The authors thank John Tebou for interesting conversations on the subject and Louisa Thomas for improving the style of presentation.

\bibliographystyle{plain}
\bibliography{outbib}

\end{document}